\documentclass[11pt,a4paper,reqno]{amsart}
\usepackage{amsfonts}
\usepackage{graphicx}   
\usepackage{graphics}
\usepackage{epsfig}
\usepackage{amssymb}
\usepackage{amsmath}
\usepackage{anysize}
\marginsize{2.5cm}{2.5cm}{2.5cm}{2.5cm}

\newtheorem{theorem}{Theorem}[section]

\newtheorem{lemma}{Lemma}[section]
\newtheorem{remark}{Remark}[section]

\newtheorem{definition}{Definition}[section]
\newtheorem{example}{Example}[section]
\newcounter{Ex} 
\newcounter{Example-ENM}
\newcounter{Example-A} 
\newcounter{Example-B} 

\def\({\left(}
\def\){\right)}

\newlength\savedwidth



\begin{document}

\setcounter{page}{0}
\bigskip
\bigskip
\title [TG Zhao, ZY Zhao, CP Li, DX Li: Spectral approximation of $\psi$-FDEs ]
{Spectral approximation of $\psi$-fractional differential equation based on mapped Jacobi functions}

\author[TG Zhao, ZY Zhao, CP Li, DX Li: Spectral approximation of $\psi$-FDEs ] {Tinggang Zhao$^{1}$, Zhenyu Zhao$^{1}$, Changpin Li$^{2}$, Dongxia Li$^{2}$}
\thanks{$^1$School of Mathematics and Statistics, Shandong University of Technology, Zibo, 255000, China\\
\indent $^{2}$Department of Mathematics, Shanghai University, Shanghai, 200444, China \\
 \indent \,\,\, e-mail:zhtg@sdut.edu.cn(T.Z); wozitianshanglai@163.com(Z.Z); lcp@shu.edu.cn(C.L); lidongxia96@163.com(D.L)\\
\indent \,\, \em Manuscript received XXXXXXXX 20XX}

\begin{abstract}
Fractional calculus with respect to function $\psi$, also named as $\psi$-fractional calculus, generalizes the Hadamard and the Riemann-Liouville fractional calculi, which causes challenge in numerical treatment. In this paper we study spectral-type methods using mapped Jacobi functions (MJFs) as basis functions and obtain efficient algorithms to solve $\psi$-fractional differential equations. In particular, we setup the Petrov-Galerkin spectral method and spectral collocation method for initial and boundary value problems involving $\psi$-fractional derivatives. We develop basic approximation theory for the MJFs and conduct the error estimates of the derived methods. We also establish a recurrence relation to evaluate the collocation differentiation matrix for implementing the spectral collocation algorithm.
Numerical examples confirm the theoretical results and demonstrate the effectiveness of the spectral and collocation methods.


\end{abstract}
\maketitle

\smallskip
\section{Introduction}\label{sec:1}
In recent decades, fractional calculus has gained more and more attention due to
its extensive applications in the various fields
of applied sciences, e.g. mechanics, microchemistry, engineering, biology, and computer science \cite{DenHWX20,Hil00,KilST06}. Till now,
 there exist several different definitions of fractional calculus, such as
 Gr\"{u}nwald-Letnikov, Riemann-Liouville, Caputo, Riesz and Hadamard
 derivatives/integrals \cite{KilST06,LiC19,LiZ15, Pod99}. It is noticed that most of the work is devoted to the issues related to Riemann-Liouville, Caputo, and Riesz derivatives. Another interesting fractional operator, sometimes collectively called $\psi$-fractional calculus, is given by fractional integration and differentiation of a function with respect to another function. This class was first proposed and motivated by Osler in \cite{Osl70}, as a generalization of fractional calculus, which includes Hadamard and the classical Riemann-Liouville fractional calculi as its special cases. Hadamard fractional calculus (the case of $\psi=\log(x)$ in the $\psi$-fractional calculus) which was introduced in \cite{Had92} has been found to be useful in some practical problems related to mechanics and engineering, e.g., the Lomnitz logarithmic creep law of special substances \cite{GarMS17,Lom56}. Recently, Caputo-Hadamard derivative is applied to studying the dynamic evolution of COVID-19 caused by the Omicron variant via the Caputo-Hadamard fractional SEIR model \cite{CaiKL22}.

The existing numerical methods to deal with the $\psi$-fractional derivative are much less than those to solve fractional differential equations (FDEs) with the Riemann-Liouville, Caputo and Riesz derivatives (see \cite{DieFFL05,DinLY17,LiC19,LiZ15,LiZL12,ZenLLBTA14} and an exhaustive list of the references). Almeida et al. \cite{AlmMM18} studied the existence and uniqueness of the solution to the initial value problem of $\psi$-Caputo-type nonlinear FDEs and developed the Picard iteration method for solving numerically the problem.
Li et al. \cite{LiLW20} considered the Caputo-Hadamard fractional partial differential equation with finite difference method and the local discontinuous Galerkin method. Fan et al. \cite{FanLL22} proposed the numerical formulas for approximating the Caputo-Hadamard fractional derivatives. More recently, Zaky et al. \cite{ZakHS22} applied the logarithmic Jacobi collocation method for the Caputo-Hadamard fractional nonlinear ordinary differential equations. Zhao et al. \cite{ZhaLL23} presented an efficient spectral collocation method for the Caputo-Hadamard FDEs. Fan et al. \cite{FanLS23} studied the $\psi$-Caputo derivative and presented several numerical schemes based on the finite difference method to discretize the fractional derivative. On the other hand, Li and Li \cite{LiL21} studied the stability of the solution to the Hadamard FDE,  where the decay law of the solution is also established. Li and Li \cite{LiL22} obtained the finite time blow-up and global existence of the solution to Cauchy problem of the semilinear time-space fractional diffusion equation with time $\psi$-Caputo derivative. Li and Li \cite{LiL22JMS} studied the blow-up of the solution to a semilinear time-space fractional diffusion equation, where the time derivative is the Caputo-type derivative with the exponential kernel, called the exponential
Caputo derivative.

From a series of works on Hadamard type derivatives \cite{FanLL22,LiL21,LiLW20,ZakHS22,ZhaLL23}, approximating this kind of fractional derivatives consults logarithmic function. In fact, applying such an idea to Caputo-type fractional differential equations and Volterra integral equations with weak singularities may achieve nice results \cite{CheS22,CheS23,CheSZZ20,YanT21}. In this paper, motivated by the above cited works, we aim to develop a high accuracy spectral and spectral collocation method which uses the mapped Jacobi functions (MJFs) as the basis functions for solving fractional differential equations with $\psi$-fractional derivatives. The main contributions of this work are as follows:
\begin{enumerate}
  \item [$\bullet$]We prove the well-posedness of four kinds of the initial and boundary value problems of the $\psi$-fractional ordinary differential equation.
  \item [$\bullet$]We define the MJFs and derive the approximate results of orthogonal projection and interpolation based on the MJFs.
  \item [$\bullet$]We conduct the error estimates of the presented Petrov-Galerkin numerical schemes based on the approximate results.
  \item [$\bullet$]We derive the differentiation matrix of the fractional $\psi$-fractional derivative and give a fast and stable evaluation method based on recurrence relationship.
  \item [$\bullet$]{We confirm the theoretical results and demonstrate the effectiveness of the proposed methods by solving the initial and boundary value problems.}
\end{enumerate}

The paper is organized as follows. In Section \ref{sec:2}, we recall the definitions and properties of $\psi$-fractional calculus and present the well-posedness of the initial and boundary problems involving $\psi$-fractional derivatives. The definition of the MJFs is given in Section \ref{sec:3}. We also build its properties and approximate results in the section. We analyze the Petrov-Galerkin spectral method and spectral collocation method in Section \ref{sec:4}. We consider in parallel the initial value problems and the boundary value problems. Numerical experiments are presented in Section \ref{sec:5}. The theoretical results about the presented approaches are confirmed by the numerical experiments, as well as the effectiveness of the presented approaches. Some conclusions are displayed in Section \ref{sec:6}.

Throughout this paper, $\mu$ is always a positive real number, $a,b$ are two real numbers satisfying $a<b$. $\psi$ is an increasing real $C^m(m=\lceil\mu\rceil)$ function on $[a,b]$ such that $\psi'>0$ for all $x\in [a,b]$. Denote $\psi_a=\psi(a), \psi_b=\psi(b),\kappa=\frac{2}{\psi_b-\psi_a}$. We use the
mapping
\begin{equation}\label{map_all}
	s(x):=\kappa(\psi(x)-\psi_a)-1=-\kappa(\psi_b-\psi(x))+1, \quad x\in[a,b]
\end{equation}
to map $[a,b]$ to $I:=[-1,1]$. It is
easy to know
\begin{equation}\label{tranf}
	{\rm{d}}s=\kappa \psi'(x){\rm{d}}x.
\end{equation}
Symbol ${{\delta}_\psi^{n}}$ ($n$, a non-negative integer number) is used for a generalized differential operator defined by
\begin{equation}\label{delta-de}
 {{\delta}^{0}_\psi} f(x)=f(x),\quad {{\delta }^{1}_\psi}=\frac{1}{\psi'(x)}\frac{{\rm{d}}}{{\rm{d}}x},\quad {{\delta}^{n}_\psi}={{\delta }^{1}_\psi}\left({{\delta}^{n-1}_\psi }\right).
\end{equation}

\section{Preliminaries}\label{sec:2}
\subsection{The $\psi$-fractional calculus} \label{subsec:2.1}

 Let $p\geq1$ and $L^p_\psi(a,b)$ is the $p$-integrable space which includes the Lebesgue measurable functions on the finite interval $(a,b)$ with respect to measure ${\rm d}\psi$, i.e.,
$$ L^p_\psi(a,b):=\left\{ u:  \int_a^b |u(x)|^p{\rm d}\psi<\infty\right\}.$$
$L^p_\psi(a,b)$ is a Banach space with norm
$$\|u\|_{L^p_\psi}=\left(\int_a^b|u(x)|^p{\rm d}\psi\right)^{1/p}.$$
When $p=2$, space $L^2_\psi(a,b)$ turns out a Hilbert space with inner product
$$(u,v)_{\psi}=\int_a^b u(x)v(x){\rm d}\psi.$$
We also use $L^2_\omega(a,b)$ for weighted Hilbert space with inner product
$$(u,v)_{\omega}=\int_a^b u(x)v(x)\omega(x){\rm d}x,$$
and norm $\|u\|_{\omega}$ as usual.

Denote by $C[a,b]$ all the continuous functions in $[a,b]$ and $AC[a,b]$ all the absolutely continuous functions in $[a,b]$.  For $n \in \mathbb{N}^+$, we define $AC^n_\psi$ as
	\begin{equation*}
		AC_\psi^n[a,b]:=\left\{u:[a,b]\rightarrow \mathbb{R},\quad {\delta_{\psi}^n } u(x)\in AC[a,b]\right\}.
	\end{equation*}

\begin{definition}\cite{FahFRS21, KilST06}\label{def_psiI} The left-sided and right-sided $\psi$-fractional integral of a given function $f(x)$ with order $\mu > 0$ are defined, respectively, as
\begin{equation*}
\begin{split}
  {_\psi{\rm D}_{a,x}^{-\mu}}f(x)=& \frac{1}{\Gamma(\mu)}\int_a^x\psi'(\tau)\left(\psi(x)-\psi(\tau)\right)
  ^{\mu-1}f(\tau){\rm d}\tau,\quad x\in (a,b),\\
   {_\psi{\rm D}_{x,b}^{-\mu}}f(x)=& \frac{1}{\Gamma(\mu)}\int_x^b\psi'(\tau)\left(\psi(\tau)-\psi(x)\right)
  ^{\mu-1}f(\tau){\rm d}\tau,\quad x\in (a,b).
\end{split}
\end{equation*}
\end{definition}

In the above definition, the condition $f(x)\in L^1_\psi(a,b)$ is often presumed albeit sufficient but applicable.

We remark the choices of $\psi$ in the above definition as follows.
\begin{remark}\label{Remark_psi}The fractional operators ${_\psi{\rm D}_{a,x}^{-\mu}}$ and ${_\psi{\rm D}_{x,b}^{-\mu}}$ given by Definition \ref{def_psiI} generalize several existing fractional integral operators:
\begin{enumerate}
  \item [$\bullet$] $\psi(x)=x$ is the classical Riemann-Liouville-type fractional integral.
  \item [$\bullet$] $\psi(x)=\log(x)$ is the Hadamard fractional integral  \cite{FanLL22,LiL21,LiLW20,ZakHS22}.
  \item [$\bullet$] $\psi(x)=\exp(x)$ is the fractional integral with exponential kernel \cite{LiL22JMS}.
\end{enumerate}
\end{remark}

\begin{lemma}\label{lmm-prp-Int-SG1} \cite{FahFRS21} The following semigroup property is valid: for $\mu,\nu>0$
\begin{equation*}
  {_\psi{\rm D}^{-\mu}_{a,x}} ({_\psi{\rm D}^{-\nu}_{a,x}}f(x))={_\psi{\rm D}^{-(\mu+\nu)}_{a,x}}f(x),\quad
   {_\psi{\rm D}^{-\mu}_{x,b}} ({_\psi{\rm D}^{-\nu}_{x,b}}f(x))={_\psi{\rm D}^{-(\mu+\nu)}_{x,b}}f(x).
\end{equation*}
\end{lemma}

\begin{lemma}\label{lmm-prp-Int-FP1}\cite{AlmMM18,KilST06} The following results are valid: for $\mu>0, \gamma>-1$
\begin{equation*}
  \begin{split}
   {_\psi{\rm D}^{-\mu}_{a,x}}\left(\psi(x)-\psi_a\right)^{\gamma}&=
  \frac{\Gamma(\gamma+1)}{\Gamma(\gamma+\mu+1)}\left(\psi(x)-\psi_a\right)^{\gamma+\mu},\\
  {_\psi{\rm D}^{-\mu}_{x,b}}\left(\psi_b-\psi(x)\right)^{\gamma}&=
  \frac{\Gamma(\gamma+1)}{\Gamma(\gamma+\mu+1)}\left(\psi_b-\psi(x)\right)^{\gamma+\mu}.
  \end{split}
\end{equation*}
\end{lemma}

\begin{definition}\label{def_HD}\cite{FahFRS21} The left-sided and right-sided $\psi$-Riemann-Liouville fractional derivative (simply, $\psi$-derivative)
of a given function $f(x)$ with order $\mu > 0$ are defined for $x\in[a,b]$, respectively, as
\begin{equation*}
\begin{split}
  {_\psi{\rm D}^{\mu}_{a,x}}f(x)=& {\delta^n_\psi}\left[{_\psi{\rm D}^{-(n-\mu)}_{a,x}}f(x)\right]
  =\frac{1}{\Gamma(n-\mu)}{\delta^n_\psi} \int_a^x\psi'(\tau)\left(\psi(x)-\psi(\tau)\right)
  ^{n-\mu-1}f(\tau){\rm{d}}\tau,\\
   {_\psi{\rm D}^{\mu}_{x,b}}f(x)=& (-1)^n{\delta^n_\psi}\left[{_\psi{\rm D}^{-(n-\mu)}_{x,b}}f(x)\right]
  =\frac{(-1)^n}{\Gamma(n-\mu)} {\delta^n_\psi} \int_x^b\psi'(\tau)\left(\psi(\tau)-\psi(x)\right)
  ^{n-\mu-1}f(\tau){{\rm{d}}\tau},\\
  \end{split}
\end{equation*}
where $\mu\in(n-1, n), n\in\mathbb{N}^+.$
\end{definition}

\begin{definition}\label{def_CHD} \cite{FahFRS21} The left-sided and right-sided $\psi$-Caputo fractional derivative
of a given function $f(x)$ with order $\mu\in(n-1,n), n\in\mathbb{N}^+$ are defined for $x\in[a,b]$, respectively, as
\begin{equation*}
\begin{split}
  {_{C\psi}{\rm D}^{\mu}_{a,x}}f(x)=&~{_\psi{\rm D}^{-(n-\mu)}_{a,x}}\left[ {\delta^n_\psi} f(x)\right]
  =\frac{1}{\Gamma(n-\mu)} \int_a^x\psi'(\tau)\left(\psi(x)-\psi(\tau)\right)
  ^{n-\mu-1}{\delta^n_\psi} f(\tau){\rm{d}}\tau,\\
   {_{C\psi}{\rm D}^{\mu}_{x,b}}f(x)=&~ (-1)^n{_\psi{\rm D}^{-(n-\mu)}_{a,x}}\left[{\delta^n_\psi} f(x)\right]
  =\frac{(-1)^n}{\Gamma(n-\mu)} \int_x^b\psi'(\tau)\left(\psi(\tau)-\psi(x)\right)
  ^{n-\mu-1}{\delta^n_\psi} f(\tau){{\rm{d}}\tau}.
  \end{split}
\end{equation*}
\end{definition}

In Definitions \ref{def_HD} and \ref{def_CHD}, the sufficient condition $f(x)\in AC^n_\psi[a,b]$ is often assumed.

\begin{lemma}\label{lmm-link-C2H}\cite{FahFRS21}
The following link between ${_{C\psi}{\rm D}^{\mu}_{a,x}}$ and ${_\psi{\rm D}^{\mu}_{a,x}}$ holds:
\begin{equation*}
  {_{C\psi}{\rm D}^{\mu}_{a,x}}f(x)=~
  {_\psi{\rm D}^{\mu}_{a,x}}\left[f(x)-\sum_{k=0}^{n-1}
  \frac{{\delta}^{k}_{\psi}f(a)}{k!}\left(\psi(x)-\psi_a\right)^k\right],\quad x\in[a,b],
\end{equation*}
where $\mu\in(n-1, n), ~{\delta}^{k}_\psi f(a)$ exists for $k=0,1,\cdots,n-1$ due to $f(x)\in AC^n_\psi[a,b]$. Similarly, there holds for $x\in[a,b]$,
\begin{equation*}
  {_{C\psi}{\rm D}^{\mu}_{x,b}}f(x)=~{_\psi{\rm D}^{\mu}_{x,b}}\left[f(x)-\sum_{k=0}^{n-1}
  \frac{(-1)^k{\delta}^{k}_\psi f(b)}{k!}\left(\psi_b-\psi(x)\right)^k\right].
\end{equation*}
\end{lemma}

\begin{lemma}\label{lmm-prp-Inv-RL}\cite{AlmMM18,FahFRS21} For $f\in C[a,b]$ and $\mu>0$, there hold
$$  {_\psi{\rm D}^{\mu}_{a,x}}\left({_\psi{\rm D}^{-\mu}_{a,x}}f(x)\right)=f(x),\quad
{_\psi{\rm D}^{\mu}_{x,b}}\left({_\psi{\rm D}^{-\mu}_{x,b}}f(x)\right)=f(x),$$
and for $f\in AC_\psi^n[a,b]$,
\begin{equation*}
  \begin{split}
  {_\psi{\rm D}^{-\mu}_{a,x}}\left({_\psi{\rm D}^{\mu}_{a,x}}f(x)\right) =&
    f(x)
-\sum_{k=0}^{n-1}\frac{[{_\psi{\rm D}^{\mu-k-1}_{a,x}}f](a)}{\Gamma(\mu-k)}
   \left(\psi(x)-\psi_a\right)^{\mu-k-1},\\
   {_\psi{\rm D}^{-\mu}_{x,b}}\left({_\psi{\rm D}^{\mu}_{x,b}}f(x)\right) =&
    f(x)
-\sum_{k=0}^{n-1}\frac{[{_\psi{\rm D}^{\mu-k-1}_{x,b}}f](a)}{\Gamma(\mu-k)}
   \left(\psi_b-\psi(x)\right)^{\mu-k-1}.
  \end{split}
\end{equation*}
\end{lemma}

\begin{lemma}\label{lmm-prp-Inv-C}\cite{LiL23} For $f\in C[a,b]$ and $\mu>0$, there hold
$$  {_{C\psi}{\rm D}^{\mu}_{a,x}}\left({_\psi{\rm D}^{-\mu}_{a,x}}f(x)\right)=f(x),\quad
{_{C\psi}{\rm D}^{\mu}_{x,b}}\left({_\psi{\rm D}^{-\mu}_{x,b}}f(x)\right)=f(x),$$
and for $f\in AC_\psi^n[a,b]$,
\begin{equation*}
  \begin{split}
  {_\psi{\rm D}^{-\mu}_{a,x}}\left({_{C\psi}{\rm D}^{\mu}_{a,x}}f(x)\right) =&
    f(x)
-\sum_{k=0}^{n-1}\frac{[{\delta}^{k}_{\psi}f](a)}{k!}
   \left(\psi(x)-\psi_a\right)^{k},\\
   {_\psi{\rm D}^{-\mu}_{x,b}}\left({_{C\psi}{\rm D}^{\mu}_{x,b}}f(x)\right) =&
    f(x)
-\sum_{k=0}^{n-1}\frac{[(-1)^k{\delta}^{k}_{\psi}f](a)}{k!}
   \left(\psi_b-\psi(x)\right)^{k}.
  \end{split}
\end{equation*}
\end{lemma}

For some special cases of $\psi$, we can obtain different fractional derivatives.
For example, if $\psi(x)=x$, then Definitions \ref{def_HD} and \ref{def_CHD} reduce to the Riemann-Liouville and Caputo fractional derivatives of order $\mu$, respectively
(see \cite{KilST06}). In the case $\psi(x)=\log(x)$, Definitions \ref{def_HD} and \ref{def_CHD} reduce to the Hadamard and Caputo-Hadamard fractional derivatives, respectively (see \cite{FahFRS21}).

\subsection{Problems and well-posedness}\label{subsec:2.2}
The ordinary differential equations of the initial value problem (IVP) and boundary value problem (BVP) play a fundamental role in theory as well as  applications.
In this work, we concern about the following four kinds of problems listed as
\begin{enumerate}
  \setlength{\itemindent}{1.5em}
  \item [(P1)] The initial value problem with $\psi$-derivative:
\begin{equation}\label{IVP-RL-P1}
\left\{
  \begin{array}{ll}
    {_\psi{\rm D}^{\mu}_{a,x}} u(x)=f(x,u), \quad x\in(a,b], \quad n-1<\mu<n, \\
    \left[{_\psi{\rm D}^{k-(n-\mu)}_{a,x}} u\right](a)=g_{k}, \quad k=0,1,\cdots,n-1.
  \end{array}
\right.
\end{equation}
  \item [(P2)] The initial value problem with $\psi$-Caputo derivative:
\begin{equation}\label{IVP-C-P2}
  \left\{
    \begin{array}{ll}
      {_{C\psi}{\rm D}^{\mu}_{a,x}} u(x)=f(x,u), \quad x\in(a,b], \quad n-1<\mu<n,  \\
      \left[{\delta}_\psi^{k}u\right](a)=c_k, \quad k=0,1,\cdots,n-1.
    \end{array}
  \right.
\end{equation}
  \item [(P3)] The boundary value problem with $\psi$-derivative:
\begin{equation}\label{BVP-RL-P3}
  \left\{
    \begin{array}{ll}
      {_\psi{\rm D}_{a,x}^{\mu}} u(x)=f(x,u),\quad x\in(a,b),\quad 1<\mu<2,\\
      \left[{_\psi{\rm D}_{a,x}^{-(2-\mu)}}u\right](a)=
\left[{_\psi{\rm D}_{a,x}^{-(2-\mu)}}u\right](b)=0.
    \end{array}
  \right.
\end{equation}
  \item [(P4)] The boundary value problem with $\psi$-Caputo derivative:
\begin{equation}\label{BVP-C-P4}
  \left\{
    \begin{array}{ll}
      {_{C\psi}{\rm D}_{a,x}^{\mu}} u(x)=f(x,u),\quad x\in(a,b),\quad 1<\mu<2,
       \\
      u(a)=u(b)=0.
    \end{array}
  \right.
\end{equation}
\end{enumerate}
Although problems we considered here only involve the left-sided type differential operators, most of our results are also valid for the right-sided type ones.

The existence and uniqueness of the above-mentioned problems is derived in two steps: (1) The problems are converted into four equivalent integral equations. (2) Each equivalent integral equation is proved to possess a unique solution by using the fixed point theorem. Some hypotheses on $f$ are always indispensable, which are corresponding to various kinds of problems.
\begin{enumerate}
  \item [(H1)]Let $h^*>a, K>0,$ and
$$G:=\left\{(x,y): a<x\leq h^*, \left|(\psi(x)-\psi_a)^{n-\mu}y-\sum_{k=0}^{n-1}\frac{(\psi(x)-\psi_a)^{k} g_{k}}{\Gamma(\mu-n+k+1)}\right|< K ~\mbox{or}~ x=a,y\in\mathbb{R}\right\}.$$
Assume that $f:G\rightarrow\mathbb{R}$ is continuous and bounded in $G$ with $M:=\sup_{(x,z)\in G}|f(x,z)|$. Moreover, there exists a constant $C_L>0$ such that
$$|f(x,y_1)-f(x,y_2)|< C_L|y_1-y_2|,\quad \forall  (x,y_1),(x,y_2) \in G.$$
Let $\widetilde{h}$ being an arbitrary positive number satisfying the constraint
\begin{equation*}
  a<\widetilde{h}<\min\left\{\psi^{-1}\left(\left(
\frac{\Gamma(2\mu-n+1)}{\Gamma(\mu-n+1)C_L}\right)^{\frac{1}{\mu}}
+\psi_a\right),\psi^{-1}\left(\left(\frac{\Gamma(\mu+1)K}{M}\right)^{1/n}+\psi_a
\right)\right\},
\end{equation*}
and $h:=\min\{h^*,\widetilde{h}\}.$
  \item [(H2)]Let $h^*>a, K>0,$ and
$$G:=\left\{(x,y): a\leq x\leq h^*, \left|y-\sum_{k=0}^{n-1}\frac{(\psi(x)-\psi_a)^{k} c_{k}}{k!}\right|< K \right\}.$$
Assume that $f:G\rightarrow\mathbb{R}$ is continuous and bounded in $G$ with $M:=\sup_{(x,z)\in G}|f(x,z)|$. Moreover, there exists a constant $C_L>0$ such that
$$|f(x,y_1)-f(x,y_2)|< C_L|y_1-y_2|,\quad \forall  (x,y_1),(x,y_2) \in G.$$
Let $\widetilde{h}$ being an arbitrary positive number satisfying the constraint
\begin{equation*}
  a<\widetilde{h}<\min\left\{\psi^{-1}\left(\left(\frac{\Gamma(\mu+1)}{C_L}
\right)^{\frac{1}{\mu}}+\psi_a\right),\psi^{-1}\left(\left(\frac{\Gamma(\mu+1)K}{M}
\right)^{\frac{1}{\mu}}+\psi_a\right)\right\},
\end{equation*}
and $h:=\min\{h^*,\widetilde{h}\}.$
  \item [(H3)]Assume that $f:[a,b]\times\mathbb{R}\rightarrow\mathbb{R}$ is continuous and satisfies a Lipschitz condition with Lipschitz constant $C_L$ with respect to the second variable
$$|f(x,y_1)-f(x,y_2)|\leq C_L|y_1-y_2|,\quad \forall  (x,y_1),(x,y_2) \in [a,b]\times\mathbb{R},$$
where
$$C_L<\frac{\kappa^\mu \Gamma(\mu+1)}{2^{\mu-1}(2+\mu)}.$$
  \item [(H4)]Assume that $f:[a,b]\times\mathbb{R}\rightarrow\mathbb{R}$ is continuous and satisfies a Lipschitz condition with Lipschitz constant $C_L$ with respect to the second variable
$$|f(x,y_1)-f(x,y_2)|\leq C_L|y_1-y_2|,\quad \forall  (x,y_1),(x,y_2) \in [a,b]\times\mathbb{R},$$
where
$$C_L<\frac{\kappa^\mu \Gamma(\mu+1)}{2^{\mu+1}}.$$
\end{enumerate}

\begin{theorem}\label{thm-4-equival-inte}The following statements are ture:
\begin{enumerate}
  \item[(i)] Assume the hypothesis (H1) holds. The function $u(x)\in C(a,h]$ solves the IVP (P1) in (\ref{IVP-RL-P1})
if and only if $u(x)$ is the solution of the Volterra integral equation of the second kind
\begin{equation}\label{eq-Vol-int-eq1-P1}
  u(x)=\sum_{k=0}^{n-1}g_k\frac{(\psi(x)-\psi_a)^{\mu-n+k}}{\Gamma(\mu-n+k+1)}
     + \frac{1}{\Gamma{(\mu)}}\int_a^x\psi'(\tau)(\psi(x)-\psi(\tau))^{\mu-1}f(\tau,u(\tau)){\rm d}\tau.
\end{equation}
  \item [(ii)]Assume the hypothesis (H2) holds. The function $u(x)\in C[a,h]$ solves the IVP (P2) in (\ref{IVP-C-P2})
if and only if $u(x)$ is the solution of the integral equation
\begin{equation}\label{eq-Vol-int-eq2-P2}
  u(x)=\sum_{k=0}^{n-1}c_k\frac{(\psi(x)-\psi_a)^{k}}{k!}
     + \frac{1}{\Gamma{(\mu)}}\int_a^x\psi'(\tau)(\psi(x)-\psi(\tau))^{\mu-1}f(\tau,u(\tau)){\rm d}\tau.
\end{equation}
  \item [(iii)]Assume the hypothesis (H3) holds. The function $u(x)\in C[a,b]$ solves the BVP (P3) in (\ref{BVP-RL-P3})
if and only if $u(x)$ is the solution of the Fredholm integral equation
\begin{equation}\label{eq-Fhm-int-eq3-P3}
\begin{split}
  u(x) =&-\frac{\kappa(\psi(x)-\psi_a)^{\mu-1}}{2\Gamma(\mu)}
\int_{a}^{b} \psi'(\tau)(\psi_b-\psi(\tau))f(\tau,u(\tau)){\rm d}\tau\\
     & + \frac{1}{\Gamma{(\mu)}}\int_a^x\psi'(\tau)(\psi(x)-\psi(\tau))^{\mu-1}f(\tau,u(\tau)){\rm d}\tau.
\end{split}
\end{equation}
  \item [(iv)]Assume the hypothesis (H4) holds.  Then $u(x)\in C[a,b]$ is the solution of BVP (P4) in (\ref{BVP-C-P4}) if and only if $u(x)$ is the solution of the Fredholm integral equation
\begin{equation}\label{eq-Fhm-int-eq4-P4}
\begin{split}
  u(x) =&-\frac{\kappa(\psi(x)-\psi_a)}{2\Gamma(\mu)}
\int_{a}^{b} \psi'(\tau)(\psi_b-\psi(\tau))^{\mu-1}f(\tau,u(\tau)){\rm d}\tau\\
     & + \frac{1}{\Gamma{(\mu)}}\int_a^x\psi'(\tau)(\psi(x)-\psi(\tau))^{\mu-1}f(\tau,u(\tau)){\rm d}\tau.
\end{split}
\end{equation}
\end{enumerate}
\end{theorem}

\begin{proof}~(i) If $u$ is a continuous solution of the IVP (\ref{IVP-RL-P1}) then we define $z(x):=f(x,u(x))$. By assumption, $z$ is a continuous function and $z(x)={_\psi{\rm D}_{a,x}^{\mu}}u(x)=\delta_\psi^n
[{_\psi{\rm D}_{a,x}^{-(n-\mu)}}]u(x)$. Thus, ${_\psi{\rm D}_{a,x}^{-(n-\mu)}}u(x)\in C^n(a,h]$. Performing the integral operator ${_\psi{\rm D}_{a,x}^{-\mu}}$ on both sides of the first equation of (\ref{IVP-RL-P1}), making use of the initial values and Lemma \ref{lmm-prp-Inv-RL}, we have the integral equation (\ref{eq-Vol-int-eq1-P1}). Note that for $k< n$,
\begin{equation*}
\begin{split}
  {_\psi{\rm D}_{a,x}^{\mu}}(\psi(x)-\psi_a)^{\mu-n+k}&={\delta}_\psi^{n}\left[
{_\psi{\rm D}_{a,x}^{-(n-\mu)}}(\psi(x)-\psi_a)^{\mu-n+k}\right]\\
&=\frac{\Gamma(\mu-n+k+1)}{\Gamma(k+1)}{\delta}_\psi^{n}
(\psi(x)-\psi_a)^{k}=0.
\end{split}
\end{equation*}
Then, perform the derivative operator ${_\psi{\rm D}_{a,x}^{\mu}}$ on both sides of
the equation (\ref{eq-Vol-int-eq1-P1}) to obtain
\begin{equation*}
  {_\psi{\rm D}_{a,x}^{\mu}}u(x)={_\psi{\rm D}_{a,x}^{\mu}}\left[
{_\psi{\rm D}_{a,x}^{-\mu}}f(x,u(x))\right]=f(x,u).
\end{equation*}
We next need to verify the initial value conditions.
For $m=0,1,\cdots,n-1$, performing ${_\psi{\rm D}_{a,x}^{\mu-n+m}}$ on both sides of
the equation (\ref{eq-Vol-int-eq1-P1}) and letting $x$ approach to $a$, one has
\begin{equation*}
  \left[{_\psi{\rm D}_{a,x}^{\mu-n+m}}u\right](a)=g_{m}.
\end{equation*}
Then statement (i) is proved.

(ii) If $u$ is a continuous solution of the IVP (\ref{IVP-C-P2}) then we define $z(x):=f(x,u(x))$, and $z$ is a continuous function. Performing the operator ${_\psi{\rm D}_{a,x}^{-\mu}}$ on both sides of the first equation of (\ref{IVP-C-P2}), making use of the initial values and Lemma \ref{lmm-prp-Inv-C}, we have the integral equation (\ref{eq-Vol-int-eq2-P2}).

An application of the differential operator ${_\psi{\rm D}_{a,x}^{\mu}}$ to both sides of (\ref{eq-Vol-int-eq2-P2}) yields the first equation of
(\ref{IVP-C-P2}). The initial value conditions are also fulfilled.

(iii) Performing the integral operator ${_\psi{\rm D}_{a,x}^{-(2-\mu)}}$ on both sides of (\ref{eq-Fhm-int-eq3-P3}) yields
\begin{equation*}
\begin{split}
  {_\psi{\rm D}_{a,x}^{-(2-\mu)}} u(x) =&-\frac{\kappa(\psi(x)-\psi_a)}{2}
\int_{a}^{b} \psi'(\tau)(\psi_b-\psi(\tau))f(\tau,u(\tau)){\rm d}\tau\\
     & + \int_a^x\psi'(\tau)(\psi(x)-\psi(\tau))f(\tau,u(\tau)){\rm d}\tau.
\end{split}
\end{equation*}
Then we have $\left[{_\psi{\rm D}_{a,x}^{-(2-\mu)}} u\right](a)=
\left[{_\psi{\rm D}_{a,x}^{-(2-\mu)}} u\right](b)=0$.
Moreover, an application of the differential operator ${_\psi{\rm D}_{a,x}^{\mu}}$ to both sides of (\ref{eq-Fhm-int-eq3-P3}) yields
(\ref{BVP-RL-P3}). Thus, $u$ in (\ref{eq-Fhm-int-eq3-P3}) solves the boundary value problem (\ref{BVP-RL-P3}).

On the other hand,  if $u$ solves the differential equation (\ref{BVP-RL-P3}) with the condition $\left[{_\psi{\rm D}_{a,x}^{-(2-\mu)}} u\right](a)=0$, then we know from (i) in Theorem \ref{thm-4-equival-inte} that it also satisfies
$$u(x)=c(\psi(x)-\psi_a)^{\mu-1}+\frac{1}{\Gamma{(\mu)}}\int_a^x\psi'(\tau)
(\psi(x)-\psi(\tau))^{\mu-1}f(\tau,u(\tau)){\rm d}\tau$$
with some undetermined constant $c$. We take $x=b$ and utilize $\left[{_\psi{\rm D}_{a,x}^{-(2-\mu)}} u\right](b)=0$ to work out $$c=-\frac{\kappa}{2\Gamma(\mu)}\int_{a}^{b}\psi'(\tau)
(\psi(b)-\psi(\tau))f(\tau,u(\tau)){\rm d}\tau.$$
Then we obtain the Fredholm integral equation (\ref{eq-Fhm-int-eq3-P3}).

(iv) If $u$ is continuous and solves the Fredholm integral equation (\ref{eq-Fhm-int-eq4-P4}), it is clear the $u(a)=u(b)=0$ because of $\kappa=2/(\psi_b-\psi_a)$. An application of the differential operator ${_\psi{\rm D}_{a,x}^{\mu}}$ to both sides of (\ref{eq-Fhm-int-eq4-P4}) yields the first equation of (\ref{BVP-C-P4}).

Similarly, if $u$ solves the differential equation (\ref{BVP-C-P4}) with the condition $u(a)=0$, then we know from (ii) in Theorem \ref{thm-4-equival-inte} that it also satisfies
$$u(x)=d(\psi(x)-\psi_a)+\frac{1}{\Gamma{(\mu)}}\int_a^x\psi'(\tau)
(\psi(x)-\psi(\tau))^{\mu-1}f(\tau,u(\tau)){\rm d}\tau$$
with some undetermined constant $d$. We take $x=b$ and utilize $u(b)=0$ to work out $$d=-\frac{\kappa}{2\Gamma(\mu)}\int_{a}^{b}\psi'(\tau)
(\psi(b)-\psi(\tau))^{\mu-1}f(\tau,u(\tau)){\rm d}\tau.$$
Then we obtain the Fredholm integral equation (\ref{eq-Fhm-int-eq4-P4}).
\end{proof}

\begin{theorem}\label{thm-exis-ine-4}The following statements are ture:
\begin{enumerate}
  \item [(i)]Under the hypothesis (H1), the equation (\ref{eq-Vol-int-eq1-P1}) possesses a uniquely solution $u\in C(a,h]$.
  \item [(ii)]Under the hypothesis (H2), the equation (\ref{eq-Vol-int-eq2-P2}) possesses a uniquely solution $u\in C[a,h]$.
  \item [(iii)]Under the hypothesis (H3), the equation (\ref{eq-Fhm-int-eq3-P3}) possesses a uniquely solution $u\in C[a,b]$.
  \item [(iv)]Under the hypothesis (H4), the equation (\ref{eq-Fhm-int-eq4-P4}) possesses a uniquely solution $u\in C[a,b]$.
\end{enumerate}
\end{theorem}
\begin{proof} (i) We define the set
$$B:=\left\{v\in C(a,h]: \sup_{a< x\leq h}\left|(\psi(x)-\psi_a)^{n-\mu}v(x)-\sum_{k=0}^{n-1}\frac{g_k(\psi(x)-\psi_a)^{k}}
{\Gamma(\mu-n+k+1)}\right|<K\right\}$$
and on this set we define the operator ${\rm A}$ by
$$ {\rm A}u:=  \sum_{k=0}^{n-1}g_k\frac{(\psi(x)-\psi_a)^{\mu-n+k}}{\Gamma(\mu-n+k+1)} + \frac{1}{\Gamma{(\mu)}}\int_a^x\psi'(\tau)(\psi(x)-\psi(\tau))^{\mu-1}f(\tau,u(\tau)){\rm d}\tau.$$
Then for any $u\in B, {\rm A}u\in C(a,h].$ Moreover, for $x\in(a,h]$, we have
\begin{equation*}
  \begin{split}
&\left|(\psi(x)-\psi_a)^{n-\mu}{\rm A} u(x)-\sum_{k=0}^{n-1}
\frac{g_k(\psi(x)-\psi_a)^k}{\Gamma(\mu-n+k+1)}\right|\\
=&\left|\frac{(\psi(x)-\psi_a)^{n-\mu}}{\Gamma(\mu)}
\int_a^x\psi'(\tau)(\psi(x)-\psi(\tau))^{\mu-1}f(\tau,u(\tau)){\rm d}\tau\right|\\
\leq &\frac{(\psi(x)-\psi_a)^{n-\mu}}{\Gamma(\mu)} M \int_a^x\psi'(\tau)(\psi(x)-\psi(\tau))^{\mu-1}{\rm d}\tau
=\frac{(\psi(x)-\psi_a)^{n}M}{\Gamma(\mu+1)}< K.
\end{split}
\end{equation*}
This shows that ${\rm A}u\in B$ if $u\in B$, i.e., the operator ${\rm A}$ maps the set $B$ into itself.

We introduce a new set
$$\widehat{B}:=\left\{v\in C(a,h]: \sup_{a<x\leq h}\left|(\psi(x)-\psi_a)^{n-\mu}v(x)\right|<\infty\right\},$$
and on this set we define a norm $\|\cdot\|_{\widehat{B}}$ by
$$\|v\|_{\widehat{B}}:= \sup_{a<x\leq h}\left|(\psi(x)-\psi_a)^{n-\mu}v(x)\right|.$$
Then $\widehat{B}$ is a normed linear space and $B$ is a complete subset of this space. We use ${\rm A}$ to rewrite the integral equation (\ref{eq-Vol-int-eq1-P1})
as ${\rm A}u=u$. Hence, in order to prove the desired result, it is sufficient to show that the operator ${\rm A}$ has a unique fixed point.

Let $v_1,v_2\in B$, for $j\geq1$ we have
\begin{equation*}
\begin{split}
  &\|{\rm A}^jv_1-{\rm A}^jv_2\|_{\widehat{B}}=\sup_{a<x\leq h}\left| (\psi(x)-\psi_a)^{n-\mu}({\rm A}{\rm A}^{j-1}v_1-{\rm A}{\rm A}^{j-1}v_2)\right|\\
=&\sup_{a<x\leq h}\left|\frac{(\psi(x)-\psi_a)^{n-\mu}}{\Gamma(\mu)}
\int_a^x\psi'(\tau)(\psi(x)-\psi(\tau))^{\mu-1}\left[f(\tau,{\rm A}^{j-1}v_1(\tau))
-f(\tau,{\rm A}^{j-1}v_2(\tau))\right]{\rm d}\tau\right|\\
\leq&\sup_{a<x\leq h}\frac{(\psi(x)-\psi_a)^{n-\mu}}{\Gamma(\mu)}
\int_a^x\psi'(\tau)(\psi(x)-\psi(\tau))^{\mu-1}\left|f(\tau,{\rm A}^{j-1}v_1(\tau))
-f(\tau,{\rm A}^{j-1}v_2(\tau))\right|{\rm d}\tau\\
\leq&\sup_{a<x\leq h}\frac{C_L(\psi(x)-\psi_a)^{n-\mu}}{\Gamma(\mu)}
\int_a^x\psi'(\tau)(\psi(x)-\psi(\tau))^{\mu-1}\left|{\rm A}^{j-1}v_1(\tau)
-{\rm A}^{j-1}v_2(\tau)\right|{\rm d}\tau\\
\leq&\frac{C_L\|{\rm A}^{j-1}v_1-{\rm A}^{j-1}v_2\|_{\widehat{B}}}{\Gamma(\mu)}
\sup_{a<x\leq h}(\psi(x)-\psi_a)^{n-\mu}\int_a^x
\psi'(\tau)(\psi(x)-\psi(\tau))^{\mu-1}(\psi(\tau)-\psi_a)^{\mu-n}{\rm d}\tau\\
=&\frac{C_L(\psi(h)-\psi_a)^{\mu}\Gamma(\mu-n+1)}{\Gamma(2\mu-n+1)}
\|{\rm A}^{j-1}v_1-{\rm A}^{j-1}v_2\|_{\widehat{B}}.
\end{split}
\end{equation*}
By induction, we conclude that
$$\|{\rm A}^jv_1-{\rm A}^jv_2\|_{\widehat{B}}\leq \left(\frac{C_L(\psi(h)-\psi_a)^{\mu}\Gamma(\mu-n+1)}
{\Gamma(2\mu-n+1)}\right)^j\|v_1-v_2\|_{\widehat{B}}.$$
From (H1), we check that
\begin{equation*}
  h<\psi^{-1}\left(\left(\frac{\Gamma(2\mu-n+1)}{\Gamma(\mu-n+1)C_L}\right)^{1/\mu}
+\psi_a\right)\Rightarrow\psi(h)-\psi_a
<\left(\frac{\Gamma(2\mu-n+1)}{\Gamma(\mu-n+1)C_L}\right)^{1/\mu}.
\end{equation*}
This ensures that
\begin{equation*}
  \frac{C_L(\psi(h)-\psi_a)^{\mu}\Gamma(\mu-n+1)}{\Gamma(2\mu-n+1)}<1.
\end{equation*}
Then, we know that ${\rm A}$ is a contraction. Thus an application of the fixed point theorem yields the existence and the uniqueness of the solution of the integral equation (\ref{eq-Vol-int-eq1-P1}).

(ii)For this case, we define the set
$$B:=\left\{v\in C[a,h]: \sup_{a\leq x\leq h}\left|v(x)-\sum_{k=0}^{n-1}\frac{c_k(\psi(x)-\psi_a)^{k}}
{k!}\right|<K\right\}$$
and on this set we define the operator ${\rm A}$ by
$$ {\rm A}u:=  \sum_{k=0}^{n-1}c_k\frac{(\psi(x)-\psi_a)^{k}}{k!} + \frac{1}{\Gamma{(\mu)}}\int_a^x\psi'(\tau)(\psi(x)-\psi(\tau))^{\mu-1}f(\tau,u(\tau)){\rm d}\tau.$$
Then for any $u\in B, {\rm A}u\in C[a,h].$ Moreover, for $x\in[a,h]$, we have
\begin{equation*}
  \begin{split}
&\left|{\rm A} u(x)-\sum_{k=0}^{n-1}
\frac{c_k(\psi(x)-\psi_a)^k}{k!}\right|
=\left|\frac{1}{\Gamma(\mu)}
\int_a^x\psi'(\tau)(\psi(x)-\psi(\tau))^{\mu-1}f(\tau,u(\tau)){\rm d}\tau\right|\\
\leq &\frac{M}{\Gamma(\mu)} \int_a^x\psi'(\tau)(\psi(x)-\psi(\tau))^{\mu-1}{\rm d}\tau
=\frac{(\psi(x)-\psi_a)^{\mu}M}{\Gamma(\mu+1)}< K.
\end{split}
\end{equation*}
This shows that the operator ${\rm A}$ maps the set $B$ into itself. Let $v_1,v_2\in B$, for $j\geq1$ we have
\begin{equation*}
\begin{split}
  &\|{\rm A}^jv_1-{\rm A}^jv_2\|=\sup_{a\leq x\leq h}\left| ({\rm A}{\rm A}^{j-1}v_1-{\rm A}{\rm A}^{j-1}v_2)\right|\\
=&\sup_{a\leq x\leq h}\left|\frac{1}{\Gamma(\mu)}
\int_a^x\psi'(\tau)(\psi(x)-\psi(\tau))^{\mu-1}\left[f(\tau,{\rm A}^{j-1}v_1(\tau))
-f(\tau,{\rm A}^{j-1}v_2(\tau))\right]{\rm d}\tau\right|\\
\leq&\sup_{a\leq x\leq h}\frac{1}{\Gamma(\mu)}
\int_a^x\psi'(\tau)(\psi(x)-\psi(\tau))^{\mu-1}\left|f(\tau,{\rm A}^{j-1}v_1(\tau))
-f(\tau,{\rm A}^{j-1}v_2(\tau))\right|{\rm d}\tau\\
\leq&\sup_{a\leq x\leq h}\frac{C_L}{\Gamma(\mu)}
\int_a^x\psi'(\tau)(\psi(x)-\psi(\tau))^{\mu-1}\left|{\rm A}^{j-1}v_1(\tau)
-{\rm A}^{j-1}v_2(\tau)\right|{\rm d}\tau\\
\leq&\frac{C_L\|{\rm A}^{j-1}v_1-{\rm A}^{j-1}v_2\|}{\Gamma(\mu)}
\sup_{a\leq x\leq h}\int_a^x
\psi'(\tau)(\psi(x)-\psi(\tau))^{\mu-1}{\rm d}\tau\\
=&\frac{C_L(\psi(h)-\psi_a)^{\mu}}{\Gamma(\mu+1)}
\|{\rm A}^{j-1}v_1-{\rm A}^{j-1}v_2\|.
\end{split}
\end{equation*}
Now we conclude that
$$\|{\rm A}^jv_1-{\rm A}^jv_2\|\leq \left(\frac{C_L(\psi(h)-\psi_a)^{\mu}}
{\Gamma(\mu+1)}\right)^j\|v_1-v_2\|.$$
The hypothesis (H2) ensures that $\frac{C_L(\psi(h)-\psi_a)^{\mu}}
{\Gamma(\mu+1)}<1$. Then, we know ${\rm A}$ is a contraction. Thus an application of the fixed point theorem yields the existence and the uniqueness of the solution of the integral equation (\ref{eq-Vol-int-eq2-P2}).

(iii) For this case, we define the operator ${\rm A}$ by
\begin{equation*}
 \begin{split}{\rm A}u: =&-\frac{\kappa(\psi(x)-\psi_a)^{\mu-1}}{2\Gamma(\mu)}
\int_{a}^{b} \psi'(\tau)(\psi_b-\psi(\tau))f(\tau,u(\tau)){\rm d}\tau\\
     & + \frac{1}{\Gamma{(\mu)}}\int_a^x\psi'(\tau)(\psi(x)-\psi(\tau))^{\mu-1}f(\tau,u(\tau)){\rm d}\tau.
 \end{split}
\end{equation*}
Then for any $u\in C[a,b], {\rm A}u\in C[a,b]$. Let $v_1,v_2\in C[a,b]$, we have
\begin{equation*}
 \begin{split}
&~~\frac{\kappa(\psi(x)-\psi_a)^{\mu-1}}{2\Gamma(\mu)}
\int_{a}^{b} \psi'(\tau)(\psi_b-\psi(\tau))|f(\tau,v_1(\tau))-
f(\tau,v_2(\tau))|{\rm d}\tau\\
\leq &~~\frac{\kappa(\psi(x)-\psi_a)^{\mu-1}C_L}{2\Gamma(\mu)}
\|v_1-v_2\|\int_{a}^{b} \psi'(\tau)(\psi_b-\psi(\tau)){\rm d}\tau
\leq \frac{(\psi_b-\psi_a)^{\mu}C_L}{2\Gamma(\mu)}\|v_1-v_2\|,
\end{split}
\end{equation*}
and
\begin{equation*}
 \begin{split}
&~~\frac{1}{\Gamma(\mu)}
\int_{a}^{x} \psi'(\tau)(\psi(x)-\psi(\tau))^{\mu-1}|f(\tau,v_1(\tau))-
f(\tau,v_2(\tau))|{\rm d}\tau\\
\leq &~~\frac{C_L}{\Gamma(\mu)}\|v_1-v_2\|
\int_{a}^{x} \psi'(\tau)(\psi(x)-\psi(\tau))^{\mu-1}{\rm d}\tau
\leq \frac{(\psi_b-\psi_a)^{\mu}C_L}{\Gamma(\mu+1)}\|v_1-v_2\|.
\end{split}
\end{equation*}
So, we have
\begin{equation*}
\|{\rm A}v_1-{\rm A}v_2\|=\max_{a\leq x\leq b}\left|{\rm A}v_1-{\rm A}v_2\right|
\leq\frac{C_L(\psi_b-\psi_a)^{\mu}(\mu+2)}{2\Gamma(\mu+1)}
\|v_1-v_2\|.
\end{equation*}
The hypothesis (H3) ensures that $\frac{C_L(\psi_b-\psi_a)^{\mu}(\mu+2)}
{2\Gamma(\mu+1)}<1$. Then, we know ${\rm A}$ is a contraction. Thus an application of the fixed point theorem yields the existence and the uniqueness of the solution of the integral equation (\ref{eq-Fhm-int-eq3-P3}).

(iv) We define the operator ${\rm A}$ by
\begin{equation*}
 \begin{split}{\rm A}u: =&-\frac{\kappa(\psi(x)-\psi_a)}{2\Gamma(\mu)}
\int_{a}^{b} \psi'(\tau)(\psi_b-\psi(\tau))^{\mu-1}f(\tau,u(\tau)){\rm d}\tau\\
     & + \frac{1}{\Gamma{(\mu)}}\int_a^x\psi'(\tau)(\psi(x)-\psi(\tau))^{\mu-1}f(\tau,u(\tau)){\rm d}\tau.
 \end{split}
\end{equation*}
Then for any $u\in C[a,b], {\rm A}u\in C[a,b]$. Let $v_1,v_2\in C[a,b]$, similar to the above process, we have
\begin{equation*}
\|{\rm A}v_1-{\rm A}v_2\|=\max_{a\leq x\leq b}\left|{\rm A}v_1-{\rm A}v_2\right|
\leq~~\frac{2C_L(\psi_b-\psi_a)^{\mu}}{\Gamma(\mu+1)}
\|v_1-v_2\|.
\end{equation*}
The hypothesis (H4) ensures that $\frac{2C_L(\psi_b-\psi_a)^{\mu}}
{\Gamma(\mu+1)}<1$. Then, we know ${\rm A}$ is a contraction. Thus an application of the fixed point theorem yields the existence and the uniqueness of the solution of the integral equation (\ref{eq-Fhm-int-eq4-P4}).
\end{proof}

\begin{theorem}\label{thm-stab-P1}Under the assumptions (H1), let
$u(x)$ and $v(x)$ be the solutions of the Cauchy type problem (\ref{IVP-RL-P1}) with different initial datum $\{g_k\}_{k=0}^{n-1}$ and $\{\widetilde{g}_k\}_{k=0}^{n-1}$. Then
\begin{equation}\label{stability-RLodes-1}
|u(x)-v(x)|\leq \sum_{k=0}^{n-1}|g_k-\widetilde{g}_k|
\left(\psi(x)-\psi_a\right)^{\mu-n+k}
E_{\mu,\mu-n+k+1}\left(C_{L}(\psi(x)-\psi_a)^\mu\right).
\end{equation}
And under the assumptions (H2), if $u(x)$ and $v(x)$ be the solutions of the Cauchy type problem (\ref{IVP-C-P2}) with different initial datum $\{c_k\}_{k=0}^{n-1}$ and $\{\widetilde{c}_k\}_{k=0}^{n-1}$,
then
\begin{equation}\label{stability-CLodes-1}
|u(x)-v(x)|\leq \sum_{k=0}^{n-1}|c_k-\widetilde{c}_k|
\left(\psi(x)-\psi_a\right)^{k}E_{\mu,k+1}\left(C_{L}(\psi(x)-\psi_a)^\mu\right),
\end{equation}
where $E_{\mu,\nu}(s)=\sum_{m=0}^{\infty}\frac{s^m}{\Gamma(\mu m+\nu)}$ is the two-parameter Mittag-Leffler function.
\end{theorem}
\begin{proof}~From Theorem \ref{thm-4-equival-inte} we know
\begin{equation*}
  \begin{split}
     |u(x)-v(x)| \leq& \sum_{k=0}^{n-1} \left|g_k-\widetilde{g}_k\right|
\frac{(\psi(x)-\psi_a)^{\mu-n+k}}{\Gamma(\mu-n+k+1)}\\
       & +\frac{C_{L}}{\Gamma(\mu)}\int_a^x\psi'(\tau)\left(\psi(x)-\psi(\tau)\right)
^{\mu-1}\left|u(\tau)-v(\tau)\right|{\rm d}\tau.
  \end{split}
\end{equation*}
Denote by $\eta(x)=|u(x)-v(x)|$, and
$$\xi(x)=\sum_{k=0}^{n-1} \left|g_k-\widetilde{g}_k\right|
\frac{(\psi(x)-\psi_a)^{\mu-n+k}}{\Gamma(\mu-n+k+1)}.$$
Then one has
\begin{equation*}
  \eta(x)\leq \xi(x)+ \frac{C_{L}}{\Gamma(\mu)} \int_a^x\psi'(\tau)\left(\psi(x)-\psi(\tau)\right)
^{\mu-1}\eta(\tau){\rm d}\tau.
\end{equation*}
From Theorem A.1 in Appendix A with $C=C_{L}/\Gamma(\mu)$, one has
$$\eta(x)\leq \xi(x)+ \int_{a}^{x}\sum_{m=1}^{\infty}\frac{(C_{L})^m}{\Gamma(m\mu)}
\psi'(\tau)(\psi(x)-\psi(\tau))^{m\mu-1}\xi(\tau){\rm d}\tau,\quad x\in [a,b].$$
Then the desired result (\ref{stability-RLodes-1}) comes from Lemma \ref{lmm-prp-Int-FP1}.

By similar way in the above proof, with replace of
$$\xi(x)=\sum_{k=0}^{n-1}|c_k-\widetilde{c}_k|\frac{(\psi(x)-\psi_a)^k}{\Gamma(k+1)},$$
the desired estimate (\ref{stability-CLodes-1}) is obtained.
\end{proof}

\begin{remark} In the case $f(x,u)= f(x)$, let
$$H^1_{0,\psi}(a,b):=\left\{v:~{\delta}^{1}_{\psi}v\in L^2_{\psi}(a,b),\quad v(a)=v(b)=0\right\}, $$
and $H^{-1}_{\psi}(a,b)$ be its dual space.
A weak form of (\ref{BVP-RL-P3}) is to find $v:={_\psi{\rm D}_{a,x}^{-(2-\mu)}}u\in H^1_{0,\psi}(a,b)$ such that
\begin{equation}\label{weak-bvp-rlpsi}
  ({\delta}^{1}_{\psi}v, {\delta}^{1}_{\psi}w)_{\psi}=-(f,w)_{\psi},\quad \forall w\in H^1_{0,\psi}(a,b).
\end{equation}
It is well known that for any $f\in H^{-1}_{\psi}(a,b)$, (\ref{weak-bvp-rlpsi}) admits a unique solution $v\in H^1_{0,\psi}(a,b)$.
Then we can recover $u$ uniquely from $u={_\psi{\rm D}_{a,x}^{2-\mu}}v$, thanks to
Lemma \ref{lmm-prp-Inv-RL}. Furthermore, we have the dependence estimate
$$\|u\|_{\psi}\leq C\|v\|_{\psi}\leq C|v|_{1,\psi}\leq C\|f\|_{\psi}.$$
\end{remark}

\section{Mapped Jacobi functions and its approximation}\label{sec:3}
\subsection{Mapped Jacobi functions (MJFs)}\label{subsec:3.1}
In this subsection, we introduce the MJFs and derive some properties for use. Denote $\mathbb{P}_N(I)$ as the space of all algebraic polynomials with degree at most $N$ defined on $I$.
Recall the Pochhammer symbol, for $c\in\mathbb{R}$ and $j\in\mathbb{N}$, which is defined by $$(c)_0=1,\quad (c)_j=:{c(c+1)\cdots(c+j-1)}=\frac{\Gamma(c+j)}{\Gamma(c)},\quad j>0.$$

The Jacobi polynomials $P^{\alpha,\beta}_n(s), s\in I$ with parameters $\alpha,\beta\in\mathbb{R}$ (\cite{Sze75}) is defined as
\begin{equation}\label{jacobi-poly-rp}
\begin{split}
  P_n^{\alpha,\beta}(s)=&\frac{(\alpha+1)_n}{n!}+\sum_{j=1}^{n-1}
  \frac{(n+\alpha+\beta+1)_j(\alpha+j+1)_{n-j}}{ j!(n-j)!}\left(\frac{s-1}{2}\right)^j\\
  &+\frac{(n+\alpha+\beta+1)_n}{n!}\left(\frac{s-1}{2}\right)^n,\quad n\geq1,
  \end{split}
\end{equation}
and $P_0^{\alpha,\beta}(s)=1$.
It is worth noting that the degenerate cases of the Jacobi polynomials when $-(\alpha+\beta+1)\in\mathbb{N}^{+}$.
For the Jacobi polynomials with one or both
parameters being negative integers, the transformation formulas are valid \cite{Sze75}:
\begin{enumerate}
  \item For $\beta\in\mathbb{R}, l\in\mathbb{N}^{+}, n\geq l\geq1$,
   \begin{equation}\label{jacobi-negative-integ1}
    P_n^{-l,\beta}(s)=c_n^{l,\beta}\left(\frac{s-1}{2}\right)^lP_{n-l}^{l,\beta}(s),
    \quad c_n^{l,\beta}=\frac{(n-l)!(\beta+n-l+1)_l}{n!}.
  \end{equation}
  \item For $\alpha\in\mathbb{R}, m\in\mathbb{N}^{+}, n\geq m\geq1$,
  \begin{equation}\label{jacobi-negative-integ1}
    P_n^{\alpha,-m}(s)=c_n^{m,\alpha}\left(\frac{s+1}{2}\right)^mP_{n-m}^{\alpha,m}(s).
  \end{equation}
  \item For $l, m\in\mathbb{N}^{+}, n\geq l+m\geq2$,
  \begin{equation}\label{jacobi-negative-integ1}
    P_n^{-l,-m}(s)=\left(\frac{s-1}{2}\right)^l\left(\frac{s+1}{2}\right)^mP_{n-l-m}^{l,m}(s).
  \end{equation}
\end{enumerate}

The well-known three-term recurrence relationship of the Jacobi polynomials $P^{\alpha,\beta}_n(s)$ with parameters $\alpha,\beta\in\mathbb{R}$ is fulfilled
for $-(\alpha+\beta+1)\notin\mathbb{N}^{+}$:
\begin{equation}\label{three-term-jacobi}
\begin{split}
  P_{n+1}^{\alpha,\beta}(s)&=(A_n^{\alpha,\beta} s-B_n^{\alpha,\beta})
  P_{n}^{\alpha,\beta}(s)-C_n^{\alpha,\beta}P_{n-1}^{\alpha,\beta}(s),\quad n\geq1,\\
  P_{0}^{\alpha,\beta}(s)&=1, \quad P_{1}^{\alpha,\beta}(s)=\frac{\alpha+\beta+2}{2}s+\frac{\alpha-\beta}{2},
  \end{split}
\end{equation}
where
\begin{equation*}
\left\{
  \begin{array}{ll}
    A_n^{\alpha,\beta}&=\displaystyle\frac{(2n+\alpha+\beta+1)(2n+\alpha+\beta+2)}
{2(n+1)(n+\alpha+\beta+1)}, \\
   B_n^{\alpha,\beta}&=\displaystyle\frac{(\beta^2-\alpha^2)(2n+\alpha+\beta+1)}
   {2(n+1)(n+\alpha+\beta+1)(2n+\alpha+\beta)}, \\
    C_n^{\alpha,\beta}&=\displaystyle\frac{(n+\alpha)(n+\beta)(2n+\alpha+\beta+2)}
{(n+1)(n+\alpha+\beta+1)(2n+\alpha+\beta)}.
  \end{array}
\right.
\end{equation*}

For $\alpha,\beta>-1$, the classical Jacobi polynomials are orthogonal with respect to
the Jacobi weight function $\omega^{\alpha,\beta}(s)=(1-s)^\alpha(1+s)^\beta$, namely,
\begin{equation}\label{Jacobi-orth}
  \int_{-1}^1P^{\alpha,\beta}_n(s)P^{\alpha,\beta}_m(s)\omega^{\alpha,\beta}(s){\rm d}s
=\gamma_n^{\alpha,\beta}\delta_{mn},
\end{equation}
where
\begin{equation*}
  \gamma_n^{\alpha,\beta}=\frac{2^{\alpha+\beta+1}\Gamma(n+\alpha+1)\Gamma(n+\beta+1)}
{(2n+\alpha+\beta+1)n!\Gamma(n+\alpha+\beta+1)},
\end{equation*}
and $\delta_{mn}$ is the Kronecker symbol, i.e.,
\begin{equation*}
  \delta_{mn}=\left\{
                \begin{array}{ll}
                  1, & \hbox{if}\quad  m=n, \\
                  0, & \hbox{otherwise.}
                \end{array}
              \right.
\end{equation*}
Some additional properties of Jacobi polynomials can be referred to \cite{SheTW11,Sze75}.

\begin{definition}\label{MJFs}Define the mapped Jacobi functions (MJFs) by
\begin{equation*}
  J^{\alpha,\beta}_{n,\psi}(x)=:P^{\alpha,\beta}_{n}(s(x)) =
  P^{\alpha,\beta}_{n}\left(\kappa(\psi(x)-\psi_a)-1\right),
\quad n=0,1,\cdots,
\end{equation*}
for $x\in[a,b].$
\end{definition}

Some useful properties are given in the following.
\begin{theorem}\label{MJFs-3-term-recu}The MJFs are generated by the three-term recurrence relation as
\begin{eqnarray}\label{P1-MJFs}
 \nonumber  J^{\alpha,\beta}_{n+1,\psi}(x)&=&\left(A_n^{\alpha,\beta}\kappa(\psi(x)-\psi_a)-
 (A_n^{\alpha,\beta}+B_n^{\alpha,\beta})
   \right)J^{\alpha,\beta}_{n,\psi}(x)-C_n^{\alpha,\beta}J^{\alpha,\beta}_{n-1,\psi}(x) \\
 \nonumber    &=& \left( -A_n^{\alpha,\beta}\kappa(\psi_b-\psi(x))+(A_n^{\alpha,\beta}-B_n^{\alpha,\beta})
   \right)J^{\alpha,\beta}_{n,\psi}(x)
   -C_n^{\alpha,\beta}J^{\alpha,\beta}_{n-1,\psi}(x),\\
  J^{\alpha,\beta}_{1,\psi}(x)&=& A_0^{\alpha,\beta}\kappa(\psi(x)
-\psi_a)-(A_0^{\alpha,\beta}+B_0^{\alpha,\beta})\\
\nonumber  &=&-A_0^{\alpha,\beta}\kappa(\psi_b-\psi(x))+(A_0^{\alpha,\beta}-B_0^{\alpha,\beta}),\\
\nonumber  J^{\alpha,\beta}_{0,\psi}(x)&=&1.
\end{eqnarray}
\end{theorem}
\begin{proof}~The three-term recurrence relation (\ref{P1-MJFs}) is a straightforward result from the corresponding three-term recurrence relation of Jacobi polynomials with the variable transform (\ref{map_all}).
\end{proof}

\begin{theorem}\label{props-MJFs} The MJFs fulfill the derivative relations:
\begin{equation}\label{P2-MJFs-1}
{\delta}^{1}_{\psi} J^{\alpha,\beta}_{n,\psi}(x)=\frac{\kappa(n+\alpha+\beta+1)}
{2}J^{\alpha+1,\beta+1}_{n-1,\psi}(x),\quad n\geq1.
\end{equation}
Furthermore, we have for $n\geq k$,
\begin{equation}\label{P2-MJFs-11}
{\delta}^{k}_{\psi}J^{\alpha,\beta}_{n,\psi}(x)=d_{n,k}^{\alpha,\beta}
J^{\alpha+k,\beta+k}_{n-k,\psi}(x),\quad d_{n,k}^{\alpha,\beta}=\frac{\kappa^k \Gamma(n+k+\alpha+\beta+1)}{2^k\Gamma(n+\alpha+\beta+1)}.
\end{equation}
And for $n\geq1$,
\begin{equation}\label{P2-MJFs-2}
J^{\alpha,\beta}_{n,\psi}(x)={\delta}^{1}_{\psi}\left(a_n^{\alpha,\beta}
J^{\alpha,\beta}_{n-1,\psi}(x)+b_n^{\alpha,\beta}
J^{\alpha,\beta}_{n,\psi}(x)+c_n^{\alpha,\beta}
J^{\alpha,\beta}_{n+1,\psi}(x)\right),
\end{equation}
where
\begin{equation*}
\left\{
  \begin{array}{ll}
    a_n^{\alpha,\beta}&=\displaystyle\frac{-2(n+\alpha)(n+\beta)}
{\kappa(n+\alpha+\beta)(2n+\alpha+\beta)(2n+\alpha+\beta+1)}, \\
   b_n^{\alpha,\beta}&=\displaystyle\frac{2(\alpha-\beta)}{\kappa(2n+\alpha+\beta)(2n+\alpha+\beta+2)}, \\
   c_n^{\alpha,\beta}&=\displaystyle\frac{2(n+\alpha+\beta+1)}{\kappa(2n+\alpha+\beta+1)(2n+\alpha+\beta+2)}.
  \end{array}
\right.
\end{equation*}
\end{theorem}
\begin{proof}~ The equality (\ref{P2-MJFs-1}) comes from
$$\frac{{\rm d}}{{\rm d}s}P_n^{\alpha,\beta}(s)=\frac{n+\alpha+\beta+1}{2}
P_{n-1}^{\alpha+1,\beta+1}(s),\quad n\geq1,$$
and (\ref{tranf}). We can use (\ref{P2-MJFs-1}) iteratively to generate (\ref{P2-MJFs-11}).

The equality (\ref{P2-MJFs-2}) comes from
for $n\geq1$,
\begin{equation*}
  P_n^{\alpha,\beta}(s)=\frac{{\rm d}}{{\rm d}s}\left(A_nP_{n-1}^{\alpha,\beta}(s)+B_nP_{n}^{\alpha,\beta}(s)
  +C_nP_{n+1}^{\alpha,\beta}(s)\right),
\end{equation*}
where
\begin{equation*}
\left\{
  \begin{array}{ll}
   A_n&=\displaystyle\frac{-2(n+\alpha)(n+\beta)}{(n+\alpha+\beta)(2n+\alpha+\beta)(2n+\alpha+\beta+1)}, \\
B_n&= \displaystyle\frac{2(\alpha-\beta)}{(2n+\alpha+\beta)(2n+\alpha+\beta+2)}, \\
C_n&= \displaystyle\frac{2(n+\alpha+\beta+1)}{(2n+\alpha+\beta+1)(2n+\alpha+\beta+2)}.
  \end{array}
\right.
\end{equation*}
The proof is completed.
\end{proof}

\begin{theorem}For $\alpha,\beta>-1$, the MJFs are orthogonal as
\begin{equation}\label{P3-MJFs}
  \int_a^bJ^{\alpha,\beta}_{n,\psi}(x)J^{\alpha,\beta}_{m,\psi}(x)
\varpi^{\alpha,\beta}(x){\rm d}x=\gamma_n^{\alpha,\beta}\delta_{nm},
\end{equation}
where $\varpi^{\alpha,\beta}(x)=\kappa^{\alpha+\beta+1} \left(\psi_b-\psi(x)\right)^\alpha
\left(\psi(x)-\psi_a\right)^\beta\psi'(x).$
\end{theorem}
\begin{proof}~ We derive from (\ref{map_all}) that
$$\int_a^bJ^{\alpha,\beta}_{n,\psi}(x)J^{\alpha,\beta}_{m,\psi}(x)
\varpi^{\alpha,\beta}(x){\rm d}x=\int_{-1}^{1}P^{\alpha,\beta}_{n}(s)P^{\alpha,\beta}_{m}(s)
\omega^{\alpha,\beta}(s){\rm d}s.$$
Then we have the orthogonality (\ref{P3-MJFs}) from (\ref{Jacobi-orth}).
\end{proof}

\begin{theorem}\label{Gauss-MJF-quad}Gauss-MJF-type quadratures hold for $\alpha,\beta>-1$. Let $\{s_j^{\alpha,\beta},\omega_j^{\alpha,\beta}\}_{j=0}^N$ be the Gauss-Jacobi-Lobatto nodes and weights. Denote
          \begin{equation*}
  \left\{x_j^{\alpha,\beta}=\psi^{-1}\left((s_j^{\alpha,\beta}+1)/{\kappa}+\psi_a\right),\quad
\varpi_j^{\alpha,\beta}=\omega_j^{\alpha,\beta}\right\}_{j=0}^N.
          \end{equation*}
Then,
\begin{equation}\label{P4-MJFs}
  \int_a^b p(x)\varpi^{\alpha,\beta}(x){\rm d}x=\sum_{j=0}^Np(x_j^{\alpha,\beta})\varpi_j^{\alpha,\beta}
,\quad \forall p(x)\in F_{2N-1}^{\psi},
\end{equation}
where
\begin{equation*}
  F_{N}^{\psi}=\mbox{\emph{span}}\left\{1,\psi(x),
\left(\psi(x)\right)^2
,\cdots, \left(\psi(x)\right)^N\right\}.
\end{equation*}
\end{theorem}
\begin{proof}~Since $p\left(\psi^{-1}(\frac{s+1}{\kappa}+\psi_a)\right)\in\mathbb{P}_{2N-1}(I)$ if $p(x)\in F^{\psi}_{2N-1}$, we can obtain from Gauss-Jacobi-Lobatto quadrature
\begin{equation*}
\begin{split}
   \int_a^b p(x)\varpi^{\alpha,\beta}(x){\rm d}x&=\int_{-1}^1p\left(\psi^{-1}\left((s+1)/{\kappa}+\psi_a\right)\right)\omega^{\alpha,\beta}(s){\rm d}s\\
&=\sum_{j=0}^Np\left(\psi^{-1}\left((s_j^{\alpha,\beta}+1)/{\kappa}+\psi_a\right)\right)\omega_j^{\alpha,\beta}=
\sum_{j=0}^Np(x_j^{\alpha,\beta})\varpi_j^{\alpha,\beta}.
\end{split}
\end{equation*}
This ends the proof.
\end{proof}

\begin{remark}\label{remarkrealparam}
The following equality is useful: For $\alpha,\beta>-1$,
\begin{equation}\label{GJquadN}
\sum_{j=0}^{N}\left(J_{N,\psi}^{\alpha,\beta}(x_j^{\alpha,\beta})\right)^2\varpi_j^{\alpha,\beta}
=\left(2+\frac{\alpha+\beta+1}{N}\right)\gamma_N^{\alpha,\beta}.
\end{equation}
\end{remark}

\subsection{Fractional calculus of the MJFs}\label{subsec:3.2}
We will derive formulas to evaluate the $\psi$-fractional calculus of the MJFs in this subsection. In effect, we have for $n=0,1,\cdots,$
\begin{equation*}
  {_\psi{\rm D}_{a,x}^{-\mu}} J_{n,\psi}^{\alpha,\beta}(x)= \kappa^{-\mu}
~{{\rm D}_{-1,s}^{-\mu}} P_{n}^{\alpha,\beta}(s),\quad
{_\psi{\rm D}_{a,x}^{\mu}} J_{n,\psi}^{\alpha,\beta}(x)= \kappa^{\mu}
~{{\rm D}_{-1,s}^{\mu}} P_{n}^{\alpha,\beta}(s),
\end{equation*}
where ${{\rm D}_{-1,x}^{-\mu}}$ and ${{\rm D}_{-1,x}^{\mu}}$ are the classical Riemann-Liouville integral and derivative defined as
$${{\rm D}_{-1,s}^{-\mu}}f(s)=\frac{1}{\Gamma(\mu)}\int_{-1}^s(s-t)^{\mu-1}f(t){\rm d}t, \quad {{\rm D}_{-1,s}^{\mu}}f(s)=\frac{{\rm d}^n}{{\rm d}s^n}\left({{\rm D}_{-1,s}^{-(n-\mu)}}f(s)\right),\quad n-1<\mu<n.$$

\begin{theorem}Let $S_{n}^{-}={_\psi{\rm D}_{a,x}^{-\mu}}J_{n,\psi}^{\alpha,\beta}(x)$ and
$S_{n}^{+}={_\psi{\rm D}_{x,b}^{-\mu}}J_{n,\psi}^{\alpha,\beta}(x)$. Then there holds the three-term recurrence relation as follows: for $n>0$,
\begin{equation*}
\begin{split}
   S_{n+1}^{-}=& \widetilde{A}_n^{\alpha,\beta}(\psi(x)-\psi_a)S_{n}^{-}-\widetilde{B}_{n,-}^{\alpha,\beta}
  S_{n}^{-}+\widetilde{C}_n^{\alpha,\beta}S_{n-1}^{-}+\widetilde{A}_n^{\alpha,\beta}{F}_{n,-}^{\alpha,\beta},\\
   S_{n+1}^{+}=& -\widetilde{A}_n^{\alpha,\beta}(\psi_b-\psi(x))S_{n}^{+}-\widetilde{B}_{n,+}^{\alpha,\beta}
  S_{n}^{+}+\widetilde{C}_n^{\alpha,\beta}S_{n-1}^{+}+\widetilde{A}_n^{\alpha,\beta}{F}_{n,+}^{\alpha,\beta},
\end{split}
\end{equation*}
with the starting terms as
\begin{equation*}
\begin{split}
   S_{0}^{-} &= \frac{(\psi(x)-\psi_a)^{\mu}}{\Gamma(\mu+1)},\quad
     S_{0}^{+} =\frac{(\psi_b-\psi(x))^{\mu}}{\Gamma(\mu+1)},\\
      S_{1}^{-} &= \frac{\kappa A_0^{\alpha,\beta}}{\Gamma(\mu+2)}(\psi(x)-\psi_a)^{\mu+1}-
  \frac{ A_0^{\alpha,\beta}+B_0^{\alpha,\beta}}{\Gamma(\mu+1)}
  (\psi(x)-\psi_a)^{\mu},\\
       S_{1}^{+} &=-\frac{\kappa A_0^{\alpha,\beta}}{\Gamma(\mu+2)}(\psi_b-\psi(x))^{\mu+1}+
  \frac{ A_0^{\alpha,\beta}-B_0^{\alpha,\beta}}{\Gamma(\mu+1)}
  (\psi_b-\psi(x))^{\mu},\\
\end{split}
\end{equation*}
where $E_n^{\alpha,\beta}=1+\mu\kappa A_n^{\alpha,\beta} c_n^{\alpha,\beta}$,
\begin{eqnarray*}
  &~& \widetilde{A}_n^{\alpha,\beta} = \frac{\kappa A_n^{\alpha,\beta}}
  {E_n^{\alpha,\beta}}, \quad
  \widetilde{B}_{n,-}^{\alpha,\beta} = -\frac{\mu\kappa A_n^{\alpha,\beta}
  b_n^{\alpha,\beta}+A_n^{\alpha,\beta}+B_n^{\alpha,\beta}}
  {E_n^{\alpha,\beta}}, \\
  &~&\widetilde{B}_{n,+}^{\alpha,\beta} = \frac{\mu\kappa A_n^{\alpha,\beta}
  b_n^{\alpha,\beta}-A_n^{\alpha,\beta}+B_n^{\alpha,\beta}}
  {E_n^{\alpha,\beta}}, \quad
  \widetilde{C}_n^{\alpha,\beta} = -\frac{\mu\kappa A_n^{\alpha,\beta}
  a_n^{\alpha,\beta}+C_n^{\alpha,\beta}}
  {E_n^{\alpha,\beta}}, \\
 &~& F_{n,-}^{\alpha,\beta}= \frac{\left(\psi(x)-\psi_a\right)^{\mu}}{\Gamma(\mu)}
\left(a_n^{\alpha,\beta}J_{n-1,\psi}^{\alpha,\beta}(a)
+b_n^{\alpha,\beta}J_{n,\psi}^{\alpha,\beta}(a)+
c_n^{\alpha,\beta}J_{n+1,\psi}^{\alpha,\beta}(a)\right),\\
&~& F_{n,+}^{\alpha,\beta}= \frac{\left(\psi_b-\psi(x)\right)^{\mu}}{\Gamma(\mu)}
\left(a_n^{\alpha,\beta}J_{n-1,\psi}^{\alpha,\beta}(b)
+b_n^{\alpha,\beta}J_{n,\psi}^{\alpha,\beta}(b)+
c_n^{\alpha,\beta}J_{n+1,\psi}^{\alpha,\beta}(b)\right),
\end{eqnarray*}
and $A_n^{\alpha,\beta},B_n^{\alpha,\beta},C_n^{\alpha,\beta}, a_n^{\alpha,\beta},b_n^{\alpha,\beta},c_n^{\alpha,\beta}$ as in Eqs. (\ref{three-term-jacobi}) and (\ref{P2-MJFs-2}).
\end{theorem}
\begin{proof}~ Noting that from (\ref{P2-MJFs-2}) $$\psi'(x)J_{n,\psi}^{\alpha,\beta}(x)=\left(a_n^{\alpha,\beta}J_{n-1,\psi}^{\alpha,\beta}(x)
+b_n^{\alpha,\beta}J_{n,\psi}^{\alpha,\beta}(x)+
c_n^{\alpha,\beta}J_{n+1,\psi}^{\alpha,\beta}(x)\right)',$$
we obtain
\begin{equation*}
  \begin{split}
      & \int_{a}^{x}\psi'(\tau)\left(\psi(x)-\psi(\tau)\right)^{\mu}
      J_{n,\psi}^{\alpha,\beta}(\tau){\rm d}\tau\\
       =&\mu\int_{a}^{x}\psi'(\tau)\left(\psi(x)-\psi(\tau)\right)^{\mu-1}
       \left(a_n^{\alpha,\beta}J_{n-1,\psi}^{\alpha,\beta}(\tau)
+b_n^{\alpha,\beta}J_{n,\psi}^{\alpha,\beta}(\tau)+
c_n^{\alpha,\beta}J_{n+1,\psi}^{\alpha,\beta}(\tau)\right){\rm d}\tau\\
       & -\left(\psi(x)-\psi_a\right)^{\mu}
       \left(a_n^{\alpha,\beta}J_{n-1,\psi}^{\alpha,\beta}(a)
+b_n^{\alpha,\beta}J_{n,\psi}^{\alpha,\beta}(a)+
c_n^{\alpha,\beta}J_{n+1,\psi}^{\alpha,\beta}(a)\right).
  \end{split}
\end{equation*}
By making use of notation $F_{n,-}^{\alpha,\beta}$,
one has from above equality
\begin{equation*}
  \begin{split}
  & {_\psi{\rm D}_{a,x}^{-\mu}}\left[(\psi(x)-\psi_a)J_{n,\psi}^{\alpha,\beta}(x)\right]\\
  &=(\psi(x)-\psi_a)S_{n}^{-}-\frac{1}{\Gamma(\mu)}\int_{a}^{x}\psi'(\tau)\left(\psi(x)-\psi(\tau)\right)^{\mu}
      J_{n,\psi}^{\alpha,\beta}(\tau){\rm d}\tau\\
      &=(\psi(x)-\psi_a)S_{n}^{-}-\mu\left(a_n^{\alpha,\beta}S_{n-1}^{-}+b_n^{\alpha,\beta}S_{n}^{-}
      +c_n^{\alpha,\beta}S_{n+1}^{-}\right)+F_{n,-}^{\alpha,\beta}.
  \end{split}
\end{equation*}
Then by the three-term recurrence relation, one has
\begin{equation*}
  \begin{split}
 S_{n+1}^{-}= & {_\psi{\rm D}_{a,x}^{-\mu}}\left[J_{n+1,\psi}^{\alpha,\beta}(x)\right]\\
=&\kappa A_n^{\alpha,\beta}~{_\psi{\rm D}_{a,x}^{-\mu}}
  \left[(\psi(x)-\psi_a)J_{n,\psi}^{\alpha,\beta}(x)\right]-(A_n^{\alpha,\beta}+B_n^{\alpha,\beta})
  S_{n}^{-}-C_n^{\alpha,\beta} S_{n-1}^{-}\\
=& \kappa A_n^{\alpha,\beta}(\psi(x)-\psi_a)S_{n}^{-}-(\mu\kappa A_n^{\alpha,\beta} b_n^{\alpha,\beta} +A_n^{\alpha,\beta}+B_n^{\alpha,\beta})S_{n}^{-}\\
&-(\mu\kappa A_n^{\alpha,\beta}a_n^{\alpha,\beta}+C_n^{\alpha,\beta})S_{n-1}^{-}
-\mu\kappa A_n^{\alpha,\beta}c_n^{\alpha,\beta} S_{n+1}^{-}+\kappa A_n^{\alpha,\beta}
F_{n,-}^{\alpha,\beta}.
  \end{split}
\end{equation*}
This is the first equality.

Similarly, one has
\begin{equation*}
  \begin{split}
      & \frac{1}{\Gamma(\mu)}\int_{x}^{b}\psi'(\tau)\left(\psi(\tau)-\psi(x)\right)^{\mu}
      J_{n,\psi}^{\alpha,\beta}(\tau){\rm d}\tau\\
       & =-\mu\left(a_n^{\alpha,\beta}S_{n-1}^{+}+
       b_n^{\alpha,\beta}S_{n}^{+}+c_n^{\alpha,\beta}S_{n+1}^{+}\right)+F_{n,+}^{\alpha,\beta},
  \end{split}
\end{equation*}
and
\begin{equation*}
  \begin{split}
  & {_\psi{\rm D}_{x,b}^{-\mu}}\left[(\psi_b-\psi(x))J_{n,\psi}^{\alpha,\beta}(x)\right]\\
  &=(\psi_b-\psi(x))S_{n}^{+}-\frac{1}{\Gamma(\mu)}\int_{x}^{b}\psi'(\tau)
  \left(\psi(\tau)-\psi(x)\right)^{\mu}J_{n,\psi}^{\alpha,\beta}(\tau){\rm d}\tau\\
      &=(\psi_b-\psi(x))S_{n}^{+}+\mu\left(a_n^{\alpha,\beta}S_{n-1}^{+}+b_n^{\alpha,\beta}S_{n}^{+}
      +c_n^{\alpha,\beta}S_{n+1}^{+}\right)-F_{n,+}^{\alpha,\beta}.
  \end{split}
\end{equation*}
Then, one has
\begin{equation*}
  \begin{split}
  S_{n+1}^{+}=&{_\psi{\rm D}_{x,b}^{-\mu}}\left[J_{n+1,\psi}^{\alpha,\beta}(x)\right]\\
=&-\kappa A_n^{\alpha,\beta}~{_\psi{\rm D}_{x,b}^{-\mu}}
  \left[(\psi_b-\psi(x))J_{n,\psi}^{\alpha,\beta}(x)\right]+(A_n^{\alpha,\beta}
  -B_n^{\alpha,\beta})S_{n}^{+}-C_n^{\alpha,\beta} S_{n-1}^{+}\\
=& -\kappa A_n^{\alpha,\beta}(\psi_b-\psi(x))S_{n}^{+}-(\mu\kappa A_n^{\alpha,\beta} b_n^{\alpha,\beta} -A_n^{\alpha,\beta}+B_n^{\alpha,\beta})S_{n}^{+}\\
&-(\mu\kappa A_n^{\alpha,\beta}a_n^{\alpha,\beta}+C_n^{\alpha,\beta})S_{n-1}^{+}
-\mu\kappa A_n^{\alpha,\beta}c_n^{\alpha,\beta} S_{n+1}^{+}+\kappa A_n^{\alpha,\beta}
F_{n,+}^{\alpha,\beta}.
  \end{split}
\end{equation*}
This is the second equality.
The starting terms can be obtained by direct calculations.
\end{proof}

\begin{theorem}\label{thmIntMJFs}For $n\geq0$, if $\alpha\in\mathbb{R},\beta>-1$, then
    \begin{equation*}
    {_\psi{\rm D}_{a,x}^{-\mu}}\left[(\psi(x)-\psi_a)^\beta J_{n,\psi}^{\alpha,\beta}(x)\right]=\frac{\Gamma(n+\beta+1)}{\Gamma(n+\beta+\mu+1)}
    (\psi(x)-\psi_a)^{\beta+\mu}J_{n,\psi}^{\alpha-\mu,\beta+\mu}(x),
  \end{equation*}
  and if $\alpha>-1,\beta\in\mathbb{R}$, then
    \begin{equation*}
    {_\psi{\rm D}_{x,b}^{-\mu}}\left[(\psi_b-\psi(x))^\alpha J_{n,\psi}^{\alpha,\beta}(x)\right]=\frac{\Gamma(n+\alpha+1)}{\Gamma(n+\alpha+\mu+1)}
    (\psi_b-\psi(x))^{\alpha+\mu}J_{n,\psi}^{\alpha+\mu,\beta-\mu}(x).
  \end{equation*}
\end{theorem}
\begin{proof}~ We derive from (\ref{map_all}) and (2.33) in \cite{CheSW16} that
\begin{equation*}
  \begin{split}
{_\psi{\rm D}_{a,x}^{-\mu}}&\left[(\psi(x)-\psi_a)^\beta J_{n,\psi}^{\alpha,\beta}(x)\right]\\
  =&\frac{1}{\Gamma(\mu)}\int_a^x\psi'(\tau)(\psi(x)-\psi(\tau))^{\mu-1}
  \left[(\psi(\tau)-\psi_a)^\beta J_{n,\psi}^{\alpha,\beta}(\tau)\right]{\rm d}\tau\\
  =&\frac{1}{\kappa^{\beta+\mu}\Gamma(\mu)}\int_{-1}^{s}(s-t)^{\mu-1}(1+t)^\beta P_n^{\alpha,\beta}(t){\rm d}t\\
  =&\frac{\Gamma(n+\beta+1)}{\Gamma(n+\beta+\mu+1)}\left(\frac{1+s}{\kappa}\right)^{\beta+\mu}
  P_n^{\alpha-\mu,\beta+\mu}(s)\\
  =&\frac{\Gamma(n+\beta+1)}{\Gamma(n+\beta+\mu+1)}
    (\psi(x)-\psi_a)^{\beta+\mu}J_{n,\psi}^{\alpha-\mu,\beta+\mu}(x).
  \end{split}
\end{equation*}
The second equality can be obtained similarly.
\end{proof}

\begin{theorem}\label{thm-fd-mjfs}
For $n\geq0$, if $\alpha\in\mathbb{R},\beta-\mu>-1$, then
    \begin{equation*}
    {_\psi{\rm D}_{a,x}^{\mu}}\left[(\psi(x)-\psi_a)^\beta J_{n,\psi}^{\alpha,\beta}(x)\right]=\frac{\Gamma(n+\beta+1)}{\Gamma(n+\beta-\mu+1)}
    (\psi(x)-\psi_a)^{\beta-\mu}J_{n,\psi}^{\alpha+\mu,\beta-\mu}(x),
  \end{equation*}
  and if $\alpha-\mu>-1,\beta\in\mathbb{R}$, then
    \begin{equation*}
    {_\psi{\rm D}_{x,b}^{\mu}}\left[(\psi_b-\psi(x))^\alpha J_{n,\psi}^{\alpha,\beta}(x)\right]=\frac{\Gamma(n+\alpha+1)}{\Gamma(n+\alpha-\mu+1)}
    (\psi_b-\psi(x))^{\alpha-\mu}J_{n,\psi}^{\alpha-\mu,\beta+\mu}(x).
  \end{equation*}
\end{theorem}
\begin{proof}~ By acting operators  ${_\psi{\rm D}_{a,x}^{\mu}}$ and ${_\psi{\rm D}_{x,b}^{\mu}}$ on the equalities in Theorem \ref{thmIntMJFs}, the desired results can be obtained from Lemma \ref{lmm-prp-Inv-RL}.
\end{proof}

The following orthogonality is also valid:
\begin{theorem}
 For $\alpha+\mu>-1, \beta-\mu>-1$,
\begin{equation}\label{orth-fd-mjofs-l}
   \int_{a}^{b}{_\psi{\rm D}_{a,x}^{\mu}}\left[(\psi(x)-\psi_a)^\beta J_{n,\psi}^{\alpha,\beta}(x)\right]
     {_\psi{\rm D}_{a,x}^{\mu}}\left[(\psi(x)-\psi_a)^\beta J_{m,\psi}^{\alpha,\beta}(x)\right]\varpi^{\alpha+\mu,-\beta+\mu}(x)dx=
\tilde{\gamma}_{n,\mu}^{\alpha,\beta}\delta_{mn}.
\end{equation}
And for $\alpha-\mu>-1,\beta+\mu>-1$,
\begin{equation}\label{orth-fd-mjofs-r}
   \int_{a}^{b}{_\psi{\rm D}_{x,b}^{\mu}}\left[(\psi_b-\psi(x))^\alpha J_{n,\psi}^{\alpha,\beta}(x)\right]
      {_\psi{\rm D}_{x,b}^{\mu}}\left[(\psi_b-\psi(x))^\alpha J_{m,\psi}^{\alpha,\beta}(x)\right]\varpi^{-\alpha+\mu,\beta+\mu}(x)dx=
\tilde{\gamma}_{n,\mu}^{\beta,\alpha}\delta_{mn},
\end{equation}
where
$$\tilde{\gamma}_{n,\mu}^{\alpha,\beta}=\frac{\kappa^{2(\mu-\beta)}2^{\alpha+\beta+1}
\Gamma^2(n+\beta+1)\Gamma(n+\alpha+\mu+1)}{(2n+\alpha+\beta+1)n!\Gamma(n+\alpha+\beta+1)
\Gamma(n+\beta-\mu+1)}.$$
\end{theorem}
\begin{proof}~ The result is derived directly from (\ref{P3-MJFs}) and Theorem \ref{thm-fd-mjfs}.
\end{proof}

Another attractive property of MJFs is that they are eigenfunctions of fractional
Sturm-Liouville-type equations. Let $w^{\alpha,\beta}(x)=(\psi_b-\psi(x))^\alpha(\psi(x)-\psi_a)^\beta$.
To show this, we define the fractional Sturm-Liouville-type operators:
\begin{equation}\label{sturm-liouvill-operators}
  \begin{split}
     ^{+}\mathcal{L}_{\alpha,\beta}^{2\mu,\psi} u(x)&:= w^{0,-\beta}(x)
{_\psi{\rm D}_{a,x}^{\mu}}\left\{w^{-\alpha+\mu,\beta+\mu}(x){_\psi{\rm D}_{x,b}^{\mu}}
w^{\alpha,0}(x) u(x)\right\},\\
      ^{-}\mathcal{L}_{\alpha,\beta}^{2\mu,\psi} u(x)&:= w^{-\alpha,0}(x)
{_\psi{\rm D}_{x,b}^{\mu}}\left\{w^{\alpha+\mu,-\beta+\mu}(x){_\psi{\rm D}_{a,x}^{\mu}}
w^{0,\beta}(x) u(x)\right\}.
  \end{split}
\end{equation}
\begin{theorem}\label{thm-sturm-liouv-1}Let $n\in\mathbb{N}^+$.
For $\alpha>\mu-1,\beta>-1$,
$$ ^{+}\mathcal{L}_{\alpha,\beta}^{2\mu,\psi} J_{n,\psi}^{\alpha,\beta}(x)=\lambda_{n,\mu}^{\alpha,\beta}J_{n,\psi}^{\alpha,\beta}(x),$$
and for $\alpha>-1,\beta>\mu-1$,
$$ ^{-}\mathcal{L}_{\alpha,\beta}^{2\mu,\psi} J_{n,\psi}^{\alpha,\beta}(x)=\lambda_{n,\mu}^{\beta,\alpha}J_{n,\psi}^{\alpha,\beta}(x),$$
where
$$\lambda_{n,\mu}^{\beta,\alpha}=\frac{\Gamma(n+\alpha+1)\Gamma(n+\beta+\mu+1)}
{\Gamma(n+\alpha-\mu+1)\Gamma(n+\beta+1)}.$$
\end{theorem}
\begin{proof}~ The desired results can be derived directly from Theorem \ref{thm-fd-mjfs}.
\end{proof}
\subsection{Approximation by the MJFs}\label{subsec:3.3}
\subsubsection{Projection approximation}\label{subsubsec:3.3.1}
We introduce the $(N+1)$-dimensional space of the MJFs as:
\begin{equation*}
  F^{\psi}_N:=\mbox{span}
\left\{J_{0,\psi}^{\alpha,\beta},J_{1,\psi}^{\alpha,\beta},\cdots,
J_{N,\psi}^{\alpha,\beta}\right\}.
\end{equation*}
It is clear that the above space of the MJFs is consistent with the space $\mbox{span}\left\{\left(\psi(x)\right)^i, i=0,1,\cdots,N\right\}$ in Theorem \ref{Gauss-MJF-quad}, so we use the same notation for it.

Let $\alpha,\beta>-1$ and $\varpi^{\alpha,\beta}(x)=\kappa^{\alpha+\beta+1} \left(\psi_b-\psi(x)\right)^\alpha
\left(\psi(x)-\psi_a\right)^\beta\psi'(x).$ For any function $u\in L^2_{\varpi^{\alpha,\beta}}(a,b)$,
denote $\mathbf{P}_N^{\alpha,\beta}: L^2_{\varpi^{\alpha,\beta}}(a,b)\rightarrow F_N^{\psi}$ the orthogonal projection such that
\begin{equation}\label{def_project_op}
  (u-\mathbf{P}_N^{\alpha,\beta} u,v)_{\varpi^{\alpha,\beta}}=0,\quad \forall v\in F_N^{\psi}.
\end{equation}

We reprensent the orthogonal projection $\mathbf{P}_N^{\alpha,\beta}$ as
\begin{equation}\label{proj_express}
  \mathbf{P}_N^{\alpha,\beta} u(x)=\sum_{k=0}^N \widehat{u}_k J_{k,\psi}^{\alpha,\beta}(x),\quad
\widehat{u}_k=\frac{\int_{a}^b u(x) J_{k,\psi}^{\alpha,\beta}(x)
\varpi^{\alpha,\beta}(x){\rm d}x}
{\|J_{k,\psi}^{\alpha,\beta}\|^2_{\varpi^{\alpha,\beta}}}
\end{equation}
with the truncated error
\begin{equation}\label{err-proj}
  u(x)-\mathbf{P}_N^{\alpha,\beta} u(x)=\sum_{k=N+1}^\infty \widehat{u}_k J_{k,\psi}^{\alpha,\beta}(x).
\end{equation}

To better describe the projection error $u(x)-\mathbf{P}_N^{\alpha,\beta}u(x)$, we need the non-uniformly mapped Jacobi-weighted Sobolev space:
\begin{equation*}
  B^m_{\alpha,\beta}(a,b):=\left\{v: {\delta}^{k}_{\psi} v\in L^2_{\varpi^{\alpha+k,\beta+k}}(a,b),\quad 0\leq k\leq m\right\}
\end{equation*}
equipped with the semi-norm and norm
\begin{equation*}
  |u|_{B^m_{\alpha,\beta}}=\|{\delta}^{m}_{\psi} u\|_{\varpi^{\alpha+m,\beta+m}}, \quad  \|u\|_{B^m_{\alpha,\beta}}=\sum_{k=0}^m\left(|u|^2_{{B^k_{\alpha,\beta}}}\right)^{1/2},
\end{equation*}
where $\|u\|_{\varpi^{\alpha,\beta}}=\left(\int_a^b|u|^2\varpi^{\alpha,\beta}(x){\rm d}x\right)^{1/2}.$
\begin{theorem}\label{err-proj-mjlofs}Let $\alpha,\beta>-1$ and $k,m,N\in\mathbb{N}$. For any function $u\in B^m_{{\alpha,\beta}}(a,b)$ and $0\leq m\leq \widetilde{m}=\min\{m,N+1\}$, we have the following estimate
\begin{equation*}
  \|{\delta}^{k}_{\psi}(u-\mathbf{P}_N^{\alpha,\beta} u)\|_{\varpi^{\alpha+k,\beta+k}}\lesssim
N^{k-\widetilde{m}}
\|{\delta}_{\psi}^{\widetilde{m},\psi}u\|_{\varpi^{\alpha+\widetilde{m},\beta+\widetilde{m}}}.
\end{equation*}
\end{theorem}
\begin{proof}~From (\ref{P2-MJFs-11}),
for $u\in L^2_{\varpi^{\alpha,\beta}}(a,b), u(x)=\sum_{j=0}^\infty \widehat{u}_j J_{j,\psi}^{\alpha,\beta}(x)$, we have
$${\delta}^{k}_{\psi} u(x)=\sum_{j=k}^\infty \widehat{u}_jd_{j,k}^{\alpha,\beta}
J_{j-k,\psi}^{\alpha+k,\beta+k}(x).$$
Then
\begin{equation*}
  \begin{split}
  &\|{\delta}^{k}_{\psi}(u-\mathbf{P}_N^{\alpha,\beta} u)\|^2_{\varpi^{\alpha+k,\beta+k}}=\sum_{j=N+1}^\infty |\widehat{u}_j|^2
\left(d_{j,k}^{\alpha,\beta}\right)^2\gamma_{j-k}^{\alpha+k,\beta+k}\\
=&\sum_{j=N+1}^\infty |\widehat{u}_j|^2 \frac{\left(d_{j,k}^{\alpha,\beta}\right)^2\gamma_{j-k}^{\alpha+k,\beta+k}}
{\left(d_{j,\widetilde{m}}^{\alpha,\beta}\right)^2\gamma_{j-\widetilde{m}}^{\alpha+\widetilde{m},\beta+\widetilde{m}}}
\left(d_{j,\widetilde{m}}^{\alpha,\beta}\right)^2\gamma_{j-\widetilde{m}}^{\alpha+\widetilde{m},\beta+\widetilde{m}}\\
\leq&\max\left\{\frac{\left(d_{j,k}^{\alpha,\beta}\right)^2\gamma_{j-k}^{\alpha+k,\beta+k}}
{\left(d_{j,\widetilde{m}}^{\alpha,\beta}\right)^2
\gamma_{j-\widetilde{m}}^{\alpha+\widetilde{m},\beta+\widetilde{m}}}\right\}
\sum_{j=N+1}^\infty |\widehat{u}_j|^2 \left(d_{j,\widetilde{m}}^{\alpha,\beta}\right)^2 \gamma_{j-\widetilde{m}}^{\alpha+\widetilde{m},\beta+\widetilde{m}}\\
\leq& \kappa^{2(k-\widetilde{m})}
\frac{\Gamma(N-\widetilde{m}+2)\Gamma(N+k+\alpha+\beta+2)}
{\Gamma(N-k+2)\Gamma(N+\widetilde{m}+\alpha+\beta+2)}
\|{\delta}^{\widetilde{m}}_{\psi}u\|^2_{\varpi^{\alpha+\widetilde{m},\beta+\widetilde{m}}}.
\end{split}
\end{equation*}
Making use of $\frac{\Gamma(n+\nu)}{\Gamma(n+\vartheta)}\leq c n^{\nu-\vartheta}$ (see \cite{ZhaWX13}) gives the desired result.
\end{proof}

With aid of Theorem \ref{thm-fd-mjfs}, we can derive the error estimate involving fractional derivative. Let
\begin{equation*}
  \widetilde{B}^\nu_{\alpha,\beta}(a,b):=\left\{v: {_\psi{\rm D}_{a,x}^{\mu}} v\in L^2_{\varpi^{\alpha+\mu,-\beta+\mu}}(a,b),\quad 0\leq \mu\leq \nu\right\}.
\end{equation*}
\begin{theorem}\label{thm-err-est-proj-fd}Let $u\in \widetilde{B}^\nu_{\alpha,0}(a,b)$. For $0<\mu\leq\nu<1$, there holds
\begin{equation}\label{thm-est-proj-fd}
  \|{_\psi{\rm D}_{a,x}^{\mu}}(u-\mathbf{P}_N^{\alpha,0}u)\|_{\varpi^{\alpha+\mu,\mu}}\lesssim
N^{\mu-\nu} \|{_\psi{\rm D}_{a,x}^{\nu}}u\|_{\varpi^{\alpha+\nu,\nu}}.
\end{equation}
\end{theorem}
\begin{proof}~ From (\ref{err-proj}) and Theorem \ref{thm-fd-mjfs}, one has
\begin{equation*}
     {_\psi{\rm D}_{a,x}^{\mu}}(u-\mathbf{P}_N^{\alpha,0}u) =\sum_{k=N+1}^{\infty}
\widehat{u}_k\frac{\Gamma(k+1)}{\Gamma(k-\mu+1)} (\psi(x)-\psi_a)^{-\mu}
J_{k,\psi}^{\alpha+\mu,-\mu}(x).
\end{equation*}
Then
\begin{equation*}
  \begin{split}
   &\|{_\psi{\rm D}_{a,x}^{\mu}}(u-\mathbf{P}_N^{\alpha,0}u)\|^2_{\varpi^{\alpha+\mu,\mu}}
 =\sum_{k=N+1}^{\infty} |\widehat{u}_k|^2 \frac{\Gamma^2(k+1)}{\Gamma^2(k-\mu+1)} \kappa^{2\mu}\gamma_k^{\alpha+\mu,-\mu}\\
      = & \kappa^{2(\mu-\nu)}\sum_{k=N+1}^{\infty} \frac{\kappa^{2\nu}|\widehat{u}_k|^2\Gamma^2(k+1)}{\Gamma^2(k-\nu+1)} \gamma_k^{\alpha+\nu,-\nu}\frac{\Gamma(k-\nu+1)\Gamma(k+\alpha+\mu+1)}
{\Gamma(k-\mu+1)\Gamma(k+\alpha+\nu+1)}\\
\leq & \kappa^{2(\mu-\nu)}\frac{\Gamma(N-\nu+2)\Gamma(N+\alpha+\mu+2)}
{\Gamma(N-\mu+2)\Gamma(N+\alpha+\nu+2)}\sum_{k=N+1}^{\infty} \frac{\kappa^{2\nu}|\widehat{u}_k|^2\Gamma^2(k+1)}{\Gamma^2(k-\nu+1)} \gamma_k^{\alpha+\nu,-\nu}\\
\leq & C\kappa^{2(\mu-\nu)}N^{2(\mu-\nu)}
\|{_\psi{\rm D}_{a,x}^{\nu}}u\|^2_{\varpi^{\alpha+\nu,\nu}}.
  \end{split}
\end{equation*}
Then we complete the proof.
\end{proof}

Let $-1<\rho<1$ and
$$E_N^{\rho,\psi}:=(\psi(x)-\psi_a)^{-\rho}F_N^\psi=\left\{
(\psi(x)-\psi_a)^{-\rho}\phi | \phi\in F_N^{\psi}\right\}.$$
For $u\in L^2_{\varpi^{\rho,\rho}}(a,b)$, define a projection operator
$\widetilde{\mathbf{P}}_N^{\rho,\rho}:L^2_{\varpi^{\rho,\rho}}(a,b)\rightarrow E_N^{\rho,\psi}$ as
\begin{equation}\label{L22-projction}
  (u-\widetilde{\mathbf{P}}_N^{\rho,\rho}u, v)_{\varpi^{\rho,\rho}}=0, \quad \forall v\in
  E_N^{\rho,\psi}.
\end{equation}
Then we have the following error estimate.
\begin{theorem}\label{err-est-L22pj} Let $l,m\in\mathbb{N}^+, l\leq m ~\mbox{and}~ \rho>0$. If $u\in L^2_{\varpi^{\rho,\rho}}(a,b)$ and ${_\psi{\rm D}_{a,x}^{-\rho}}u\in B_{0,0}^{m}(a,b)$, then for $1\leq m\leq N+1$,
\begin{equation*}
  \|{_\psi{\rm D}_{a,x}^{l-\rho}}(u-\widetilde{\mathbf{P}}_N^{\rho,\rho}u)\|_{\varpi^{l,l}}\lesssim
N^{l-m} \|{_\psi{\rm D}_{a,x}^{m-\rho}}u\|_{\varpi^{m,m}}.
\end{equation*}
\end{theorem}
\begin{proof}~ Since the set $\left\{(\psi(x)-\psi_a)^{-\rho}J_{j,\psi}^{\rho,-\rho}\right\}_{j=0}^{\infty}$ is a complete orthogonal basis of the space $L^2_{\varpi^{\rho,\rho}}(a,b)$, we have
$$u-\widetilde{\mathbf{P}}_N^{\rho,\rho}u=\sum_{k=N+1}^{\infty}\widehat{u}_k
(\psi(x)-\psi_a)^{-\rho}J_{k,\psi}^{\rho,-\rho}(x)$$
with
$$\widehat{u}_k=\frac{\int_{a}^{b}u(x)(\psi(x)-\psi_a)^{-\rho}J_{k,\psi}^{\rho,-\rho}(x)
\varpi^{\rho,\rho}(x){\rm d}x}{\|(\psi(x)-\psi_a)^{-\rho}J_{k,\psi}^{\rho,-\rho}(x)\|^2_{\varpi^{\rho,\rho}}}.$$
Hence, by Theorems \ref{thmIntMJFs} and \ref{props-MJFs}, we obtain
\begin{equation*}
  {_\psi{\rm D}_{a,x}^{m-\rho}}(u-\widetilde{\mathbf{P}}_N^{\rho,\rho}u)=
  {\delta}^{m}_{\psi}\left[{_\psi{\rm D}_{a,x}^{-\rho}}(u-\widetilde{\mathbf{P}}_N^{\rho,\rho}u)\right]
  =\sum_{k=N+1}^{\infty}\widehat{u}_k
  \frac{\kappa^m(k+m)!\Gamma(k+\rho+1)}{2^m(k!)^2}J_{k-m,\psi}^{m,m}(x).
\end{equation*}
Set
$$h_k^{l,m}=\kappa^{2(l-m)}\frac{(k+l)!(k-m)!}
{(k+m)!(k-l)!}.$$
Then, we have
\begin{equation*}
\begin{split}
&\|{_\psi{\rm D}_{a,x}^{l-\rho}}(u-\widetilde{\mathbf{P}}_N^{\rho,\rho}u)\|^2
  _{\varpi^{l,l}}
  =\sum_{k=N+1}^{\infty}|\widehat{u}_k|^2
  \frac{\kappa^{2m}((k+m)!)^2\Gamma^2(k+\rho+1)}{2^{2m}(k!)^4}
  \gamma_{k-m}^{m,m}h_k^{l,m}\\
  &\leq\max_{k\geq N+1}\left\{h_k^{l,m}\right\}
  \|{_\psi{\rm D}_{a,x}^{m-\rho}}u\|^2_{\varpi^{m,m}}
  \leq c\kappa^{2(l-m)}N^{2(l-m)}\|{_\psi{\rm D}_{a,x}^{m-\rho}}u\|^2_{\varpi^{m,m}}.
\end{split}
\end{equation*}
The proof is completed.
\end{proof}

Consider the $H^1$ projection operator $\mathbf{P}_N^{1,\alpha,\beta}: H^1_{0,\psi}\rightarrow F_N^{\psi}\cap H^1_{0,\psi}$ such that
\begin{equation}\label{h1-projction}
  ({\delta}^{1}_{\psi}(\mathbf{P}_N^{1,\alpha,\beta}u-u), {\delta }^{1}_{\psi}v)_{\varpi^{\alpha,\beta}}=0, \quad \forall v\in F_N^{\psi}\cap H^1_{0,\psi}.
\end{equation}
\begin{theorem}\label{h1-proj-approx-res} Let $-1<\alpha,\beta<1$. If $u\in H^1_{0,\psi}(a,b)$ and ${\delta}^{1}_{\psi}u\in B_{\alpha,\beta}^{m-1}(a,b)$, then for $1\leq m\leq N+1$,
\begin{equation*}
  \|\mathbf{P}_N^{1,\alpha,\beta}u-u\|_{1,\varpi^{\alpha,\beta}}\lesssim N^{1-m}\|{\delta }^{m}_{\psi}u\|_{\varpi^{\alpha+m-1,\beta+m-1}}.
\end{equation*}
\end{theorem}
\begin{proof}~ Denote by $v(x)={\delta}^{1}_{\psi}u(x)$. Setting
$$\phi(x)=\int_a^x \mathbf{P}_{N-1}^{\alpha,\beta}v{\rm d}\psi-\frac{\psi(x)-\psi_a}{\psi_b-\psi_a}\int_a^b\mathbf{P}_{N-1}^{\alpha,\beta}v{\rm d}\psi, $$
we have $\phi\in F_N^{\psi}\cap H^1_{0,\psi}$ and
$${\delta}^{1}_{\psi}\phi=\mathbf{P}_{N-1}^{\alpha,\beta}v-\frac{\kappa}{2}\int_{a}^{b}
\mathbf{P}_{N-1}^{\alpha,\beta}v{\rm d}\psi.$$
Hence, by the triangle inequality, one has
\begin{equation}\label{ine_11}
  \|v-{\delta}^{1}_{\psi}\phi\|_{\varpi^{\alpha,\beta}}\leq \|v-\mathbf{P}_{N-1}^{\alpha,\beta}v\|_{\varpi^{\alpha,\beta}}+
  \frac{\kappa\sqrt{\gamma_0^{\alpha,\beta}}}{2}\left|
  \int_{a}^{b}
\mathbf{P}_{N-1}^{\alpha,\beta}v{\rm d}\psi\right|.
\end{equation}
Due to $u(\pm1)=0$, we derive for $-1<\alpha,\beta<1$,
$$\left|\int_{a}^{b}
\mathbf{P}_{N-1}^{\alpha,\beta}v{\rm d}\psi\right|
=\left|\int_{a}^{b}
\mathbf{P}_{N-1}^{\alpha,\beta}v-v {\rm d}\psi\right|\leq \sqrt{\gamma_0^{-\alpha,-\beta}}
\|v-\mathbf{P}_{N-1}^{\alpha,\beta}v\|_{\varpi^{\alpha,\beta}}.
$$
Then, by Theorem \ref{err-proj-mjlofs} we have
\begin{equation*}
  \|{\delta}^{1}_{\psi}(\mathbf{P}_N^{1,\alpha,\beta}u-u)\|_{\varpi^{\alpha,\beta}}\leq
  \|v-{\delta}^{1}_{\psi}\phi\|_{\varpi^{\alpha,\beta}}
  \leq c\|v-\mathbf{P}_{N-1}^{\alpha,\beta}v\|_{\varpi^{\alpha,\beta}}
 \lesssim N^{1-m}\|{\delta}^{m}_{\psi}u\|_{\varpi^{\alpha+m-1,\beta+m-1}}.
\end{equation*}
The desired estimate follows from the Poincar\'{e} inequality.
\end{proof}

\subsubsection{Interpolation approximation}\label{subsubsec:3.3.2}
Let $\{x_j^{\alpha,\beta}\}_{j=0}^N$ be the mapped Gauss-Lobatto nodes defined in (\ref{P4-MJFs}).
For function $u(x)\in C[a,b]$, we define the interpolation operator $\mathbf{I}_N^{\alpha,\beta}:C[a,b]\rightarrow F_N^{\psi}$ at given nodes $\{x^{\alpha,\beta}_j\}_{j=0}^N$ such that
\begin{equation}\label{def_interp_op}
  \mathbf{I}_N^{\alpha,\beta} u(x) = u(x),\quad x=x^{\alpha,\beta}_j,~j=0,1,\cdots,N.
\end{equation}

\begin{theorem}\label{err-interp}Let $\alpha,\beta>-1$ and $k,m,N\in\mathbb{N}$.  For $u\in B^m_{\alpha,\beta}(a,b)~(m\geq1)$ and $0\leq m\leq \widetilde{m}=\min\{m,N+1\}$, then
\begin{equation*}
  \|{\delta}^{k}_{\psi}(u-\mathbf{I}_N^{\alpha,\beta} u)\|_{\varpi^{\alpha+k,\beta+k}}\lesssim
N^{k-\widetilde{m}}
\|{\delta}^{\widetilde{m}}_{\psi}u\|_{\varpi^{\alpha+\widetilde{m},\beta+\widetilde{m}}}.
\end{equation*}
\end{theorem}
\begin{proof}~ Note that $\mathbf{I}_N^{\alpha,\beta}[\mathbf{P}_N^{\alpha,\beta}u]=\mathbf{P}_N^{\alpha,\beta}u$. Hence one has
\begin{equation*}
\|{\delta}^{k}_{\psi}(u-\mathbf{I}_N^{\alpha,\beta} u)\|_{\varpi^{\alpha+k,\beta+k}}\leq
\|{\delta}^{k}_{\psi}(u-\mathbf{P}_N^{\alpha,\beta} u)\|_{\varpi^{\alpha+k,\beta+k}}
+\|{\delta}^{k}_{\psi}[\mathbf{I}_N^{\alpha,\beta}(u-\mathbf{P}_N^{\alpha,\beta} u)]\|_{\varpi^{\alpha+k,\beta+k}}.
\end{equation*}
From Theorems A.2, A.3, and A.4 in Appendix A, we have
\begin{equation*}
  \|{\delta}^{k}_{\psi}[\mathbf{I}_N^{\alpha,\beta}(u-\mathbf{P}_N^{\alpha,\beta} u)]\|_{\varpi^{\alpha+k,\beta+k}}\lesssim N^{k-\widetilde{m}}\|{\delta }^{\widetilde{m}}_{\psi}u\|_{\varpi^{\alpha+\widetilde{m},\beta+\widetilde{m}}}.
\end{equation*}
Combining Theorem \ref{err-proj-mjlofs} with the above estimates, we obtain the desired result.
\end{proof}

\section{Spectral and collocation method based on the MJFs} \label{sec:4}
\subsection{Fractional IVP}\label{subsec:4.1}
\subsubsection{Petrov-Galerkin spectral method}\label{subsubsec:4.1.1}
For simplification, we first consider the linear IVP with $0<\mu<1$:
\begin{equation}\label{init-prob-RL}
  {_\psi{\rm D}_{a,x}^{\mu}}u(x)=f(x),\quad x\in(a,b],\quad {_\psi{\rm D}_{a,x}^{-(1-\mu)}}u(a)=g_0.
\end{equation}
We assume that $g_0=0$ without loss of generality. (If $g_0\neq0$, let  $u(x)=v(x)+g_0\frac{(\psi(x)-\psi_a)^{\mu-1}}{\Gamma(\mu)}$, then ${_\psi{\rm D}_{a,x}^{-(1-\mu)}}v(a)=0$).

It is clear that ${_\psi{\rm D}_{a,x}^{-(1-\mu)}}\phi(a)=0$ for every $\phi\in E_N^{-\mu,\psi}$.

The Petrov-Galerkin spectral scheme of (\ref{init-prob-RL}) reads as:
to find $u_N\in E_N^{-\mu,\psi}$ such that
\begin{equation}\label{Petrov-spec-Gal-ivp-rl}
  ({_\psi{\rm D}_{a,x}^{\mu}}u_N, v_N)_{\varpi^{0,0}}=(\mathbf{I}_N^{0,0}f,v_N)_{\varpi^{0,0}},\quad \forall v_N\in F_{N}^{\psi}.
\end{equation}
\begin{theorem}\label{convrate-PGS-ivp}Let $u$ and $u_N$ be the solutions of (\ref{init-prob-RL}) and (\ref{Petrov-spec-Gal-ivp-rl}), respectively. If $f\in C[a,b]\cap B^{m-1}_{-1,-1}(a,b)(m\geq1)$, then we have that for $m\leq N+1$,
\begin{equation}\label{err-est-PGs-ivp}
  \|{_\psi{\rm D}_{a,x}^{\mu}}(u-u_N)\|_{\varpi^{0,0}}+
  \|(u-u_N)\|_{\varpi^{0,0}}\lesssim N^{-m}\|{\delta}^{m}_{\psi}f\|_{\varpi^{m,m}}.
\end{equation}
\end{theorem}
\begin{proof}~ It is clear that the set of $\left\{(\psi(x)-\psi_a)^\mu J_{n,\psi}^{-\mu,\mu}(x)\right\}_{n=0}^\infty$ forms a complete orthogonal basis of $L^2_{\varpi^{-\mu,-\mu}}(a,b)$. For $u\in L^2_{\varpi^{-\mu,-\mu}}(a,b)$, the projection error $u-\widetilde{\mathbf{P}}_N^{-\mu,-\mu}u$ is
$$u-\widetilde{\mathbf{P}}_N^{-\mu,-\mu}u=\sum_{n=N+1}^{\infty}\widehat{u}_n
(\psi(x)-\psi_a)^\mu J_{n,\psi}^{-\mu,\mu}(x)$$
with
$$\widehat{u}_n=\frac{\int_{a}^{b}u(x)(\psi(x)-\psi_a)^\mu J_{n,\psi}^{-\mu,\mu}(x)\varpi^{-\mu,-\mu}(x){\rm d}x}{\|(\psi(x)-\psi_a)^\mu J_{n,\psi}^{-\mu,\mu}(x)\|^2_{\varpi^{-\mu,-\mu}}}.$$
By acting operator ${_\psi{\rm D}_{a,x}^{\mu}}$ on the projection error and utilizing Theorem \ref{thm-fd-mjfs}, we have
$${_\psi{\rm D}_{a,x}^{\mu}}(u-\widetilde{\mathbf{P}}_N^{-\mu,-\mu}u)
=\sum_{n=N+1}^{\infty}\widehat{u}_n\frac{\Gamma(n+\mu+1)}{n!}J_{n,\psi}^{0,0}(x).$$
Then we have
$$({_\psi{\rm D}_{a,x}^{\mu}}(u-\widetilde{\mathbf{P}}_N^{-\mu,-\mu}u),v)_{\varpi^{0,0}}=0,
\quad \forall v\in F_N^{\psi}.$$
Since ${_\psi{\rm D}_{a,x}^{\mu}}u=f$, it comes out
$$(f-{_\psi{\rm D}_{a,x}^{\mu}}\widetilde{\mathbf{P}}_N^{-\mu,-\mu}u,v)_{\varpi^{0,0}}=0,
\quad \forall v\in F_N^{\psi},$$
and we know ${_\psi{\rm D}_{a,x}^{\mu}}\widetilde{\mathbf{P}}_N^{-\mu,-\mu}u=
\mathbf{P}_N^{0,0}f$. We also have ${_\psi{\rm D}_{a,x}^{\mu}}u_N=\mathbf{I}_N^{0,0}f$ by
(\ref{Petrov-spec-Gal-ivp-rl}). Therefore, we get
\begin{equation}\label{relat-pp2}
  {_\psi{\rm D}_{a,x}^{\mu}}(\widetilde{\mathbf{P}}_N^{-\mu,-\mu}u-u_N)=
  \mathbf{P}_N^{0,0}f-\mathbf{I}_N^{0,0}f.
\end{equation}
Making use of the triangle inequality leads to
\begin{equation}\label{est-pp33}
  \begin{split}
&\|{_\psi{\rm D}_{a,x}^{\mu}}(u-u_N)\|_{\varpi^{0,0}}\\
\leq &
\|{_\psi{\rm D}_{a,x}^{\mu}}(u-\widetilde{\mathbf{P}}_N^{-\mu,-\mu}u)\|_{\varpi^{0,0}}
+ \|{_\psi{\rm D}_{a,x}^{\mu}}(\widetilde{\mathbf{P}}_N^{-\mu,-\mu}u-u_N)\|_{\varpi^{0,0}}\\
\leq &
\|{_\psi{\rm D}_{a,x}^{\mu}}(u-\widetilde{\mathbf{P}}_N^{-\mu,-\mu}u)\|_{\varpi^{0,0}}
+\|f-\mathbf{P}_N^{0,0}f\|_{\varpi^{0,0}}+
\|f-\mathbf{I}_N^{0,0}f\|_{\varpi^{0,0}}.
  \end{split}
\end{equation}
From approximation results in Theorems \ref{err-est-L22pj}, \ref{err-proj-mjlofs}
and \ref{err-interp}, we estimate the three items in the right-hand side of (\ref{est-pp33})
as follows:
\begin{equation*}
  \begin{split}
\|{_\psi{\rm D}_{a,x}^{\mu}}(u-\widetilde{\mathbf{P}}_N^{-\mu,-\mu}u)\|_{\varpi^{0,0}}
&\lesssim N^{-m}\|{_\psi{\rm D}_{a,x}^{m+\mu}}u\|_{\varpi^{m,m}}=
N^{-m}\|{\delta}^{m}_{\psi}f\|_{\varpi^{m,m}},\\
\|f-\mathbf{P}_N^{0,0}f\|_{\varpi^{0,0}}&\lesssim N^{-m}\|{\delta }^{m}_{\psi}f\|_{\varpi^{m,m}}, \\
\|f-\mathbf{I}_N^{0,0}f\|_{\varpi^{0,0}}&\lesssim N^{-m}\|{\delta }^{m}_{\psi}f\|_{\varpi^{m,m}}.
 \end{split}
\end{equation*}
Then, we conclude
\begin{equation}\label{err-est-s4-11}
  \|{_\psi{\rm D}_{a,x}^{\mu}}(u-u_N)\|_{\varpi^{0,0}}\lesssim N^{-m}\|{\delta }^{m}_{\psi}f\|_{\varpi^{m,m}}.
\end{equation}
From Lemma \ref{lmm-prp-Inv-RL} and the boundness of operator ${_\psi{\rm D}_{a,x}^{-\mu}}$, we obtain
\begin{equation}\label{err-est-s4-12}
  \|u-u_N\|_{\varpi^{0,0}}= \|{_\psi{\rm D}_{a,x}^{-\mu}}({_\psi{\rm D}_{a,x}^{\mu}}
  (u-u_N))\|_{\varpi^{0,0}}\leq c\|{_\psi{\rm D}_{a,x}^{\mu}}
  (u-u_N)\|_{\varpi^{0,0}}.
\end{equation}
The proof is completed by combining estimates (\ref{err-est-s4-11}) and (\ref{err-est-s4-12}).
\end{proof}

In order to implement the Petrov-Galerkin scheme (\ref{Petrov-spec-Gal-ivp-rl}), we expand the approximate solution as
\begin{equation}\label{expre-numsol-pgs}
  u_N(x)=\sum_{k=0}^{N}\widehat{u}_k(\psi(x)-\psi_a)^\mu J_{k,\psi}^{-\mu,\mu}(x).
\end{equation}
Inserting $u_N$ in  (\ref{Petrov-spec-Gal-ivp-rl}) and taking $v_N=J_{n,\psi}^{0,0}(x), n=0,1,\cdots,N$, one has
\begin{equation}\label{numsol-psgs-c}
  \widehat{u}_k=\frac{(2k+1)k!}{2\Gamma(k+\mu+1)}
  (\mathbf{I}_N^{0,0}f,J_{k,\psi}^{0,0})_{\varpi^{0,0}},\quad k=0,1,\cdots,N.
\end{equation}
We approximate the integral in the right-hand side of (\ref{numsol-psgs-c}) by the Gauss-MJF-type quadrature in Theorem \ref{Gauss-MJF-quad} and obtain the approximate numerical solution $u_N$ by (\ref{expre-numsol-pgs}).

\begin{remark} It is easy to implement the
 Galerkin spectral scheme of (\ref{init-prob-RL}), say, to find $u_N\in F_N^{\psi}$ such that
\begin{equation}\label{spec-Gal-ivp-rl}
  ({_\psi{\rm D}_{a,x}^{\mu}}u_N, v_N)_{\varpi^{\alpha+\mu,0}}=(\mathbf{I}_N^{\alpha+\mu,0}f,v_N)_{\varpi^{\alpha+\mu,0}},\quad \forall v_N\in F_{N}^{\psi}.
\end{equation}
We can obtain an error estimate for the Galerkin spectral scheme (\ref{spec-Gal-ivp-rl}) in an analogous way. The scheme (\ref{Petrov-spec-Gal-ivp-rl}) achieves the convergence rate that depends only on the regularity of $f$ (see Theorem \ref{convrate-PGS-ivp}). However, the convergence rate of
the scheme (\ref{spec-Gal-ivp-rl}) depends on the regularity of the exact solution $u$.
\end{remark}

Now we consider the IVP with $\psi$-Caputo derivative for $0\leq\mu\leq1$:
\begin{equation}\label{init-prob-C}
  {_{C\psi}{\rm D}_{a,x}^{\mu}}u(x)=f(x),\quad x\in(a,b],\quad u(a)=0.
\end{equation}
Noting the initial condition $u(a)=0$ and Lemma \ref{lmm-link-C2H}, we know
$${_{C\psi}{\rm D}_{a,x}^{\mu}}u(x)={_\psi{\rm D}_{a,x}^{\mu}}u(x).$$ Hence, the above analysis is valid also for this case by replace of $F_N^\psi$ with $F_N^\psi\cap\{\phi | \phi(a)=0\}.$
It is easy to generalize the previous analysis to the higher-order case $\mu\in(n-1,n),  n>1$. We omit the details.

\subsubsection{Spectral collocation method}\label{subsubsec:4.1.2}
For nonlinear problem, collocation method is more popular than the Galerkin method.
Let $\{x_j^{\alpha,\beta}\}_{j=0}^N$ be the collocation points as defined in (\ref{P4-MJFs}).
Now consider the nonlinear IVP (P1) with $0<\mu<1$:
\begin{equation}\label{init-prob-RL-nonl-s4}
  {_\psi{\rm D}_{a,x}^{\mu}}u(x)=f(x,u),\quad x\in(a,b],\quad {_\psi{\rm D}_{a,x}^{-(1-\mu)}}u(a)=0.
\end{equation}
We present the spectral collocation scheme for (\ref{init-prob-RL-nonl-s4}) as:
to find $u_N\in F_N^\psi$ such that
\begin{equation}\label{init-prob-RL-dis}
  {_\psi{\rm D}_{a,x}^{\mu}} u_N(x_j^{\alpha,\beta})=f(x_j^{\alpha,\beta},u_N(x_j^{\alpha,\beta})),\quad j=1,2,\cdots,N
\end{equation}
and
$$ {_\psi{\rm D}_{a,x}^{-(1-\mu)}}u_N(x_0^{\alpha,\beta})=0.$$
Let us introduce the mapped Lagrange interpolation basis functions as
\begin{equation}\label{Lagrg_basis}
  L_j(x)=\frac{\prod_{i\neq j}\psi(x)-\psi({x_i^{\alpha,\beta}})}
{\prod_{i\neq j}\psi(x_j^{\alpha,\beta})-\psi({x_i^{\alpha,\beta}})},
\quad j=0,1,\cdots,N,
\end{equation}
which satisfies $L_j(x_i^{\alpha,\beta})=\delta_{ji}$. The \emph{differentiation matrix of the $\psi$-fractional derivative (DMFD)} is given by
$$\left[\mathbf{D}^{\mu,\psi}\right]_{(N+1)\times (N+1)}:=([ {_\psi{\rm D}_{a,x}^{\mu}}L_j](x_k^{\alpha,\beta}))_{k ,j=0}^N.$$
The above matrix has to be revised according to the initial condition.

Let $\mathbf{b}=[{_\psi{\rm D}_{a,x}^{-(1-\mu)}}L_0(a),\cdots, {_\psi{\rm D}_{a,x}^{-(1-\mu)}}L_N(a)]$. We replace the first row of $\mathbf{D}^{\mu,\psi}$ with $\mathbf{b}$.
Then, using the same notation, the discrete scheme (\ref{init-prob-RL-dis}) is in matrix-vector form
\begin{equation}\label{dis-IVP-matvec}
  \mathbf{D}^{\mu,\psi}\mathbf{u}=f(\mathbf{x},\mathbf{u})
\end{equation}
with $\mathbf{x}=[x_0^{\alpha,\beta},x_1^{\alpha,\beta},\cdots,x_N^{\alpha,\beta}]^T, \mathbf{u}=u_N(\mathbf{x})$ (The first component of $f$ is set to zero). The system (\ref{dis-IVP-matvec}) is nonlinear, some kinds of iterative techniques can be employed to solve it.

\begin{remark}
To evaluate the entries of DMFD, we first note that
\begin{equation}\label{link-in-mjlofs}
  L_j(x)=\sum_{i=0}^N l_{ij}{J}_{i,\psi}^{\alpha,\beta}(x)
\end{equation}
with
\begin{equation*}
   l_{kj}=\frac{{J}_{k,\psi}^{\alpha,\beta}(x_j^{\alpha,\beta})\varpi_j}{\gamma_k^{\alpha,\beta}},\quad k=0,1,\cdots,N-1;\quad
   l_{Nj}=\frac{{J}_{N,\psi}^{\alpha,\beta}(x_j^{\alpha,\beta})\varpi_j}
{\left(2+\frac{\alpha+\beta+1}{N}\right)\gamma_N^{\alpha,\beta}}.
\end{equation*}
Then we need to evaluate $[ {_\psi{\rm D}_{a,x}^{\mu}}{J}_{j,\psi}^{\alpha,\beta}](x_k^{\alpha,\beta})$, which can be done efficiently by the relations in Section \ref{subsec:3.2}.
\end{remark}

For the nonlinear IVP (P2) with the Caputo derivative ($0<\mu<1$):
\begin{equation}\label{init-prob-RL-nonl}
  {_{C\psi}{\rm D}_{a,x}^{\mu}}u(x)=f(x,u),\quad x\in(a,b],\quad u(a)=0,
\end{equation}
the DMFD alters to
$$\left[\mathbf{D}^{\mu,\psi}\right]_{N\times N}:=([ {_{C\psi}{\rm D}_{a,x}^{\mu}}L_j](x_k^{\alpha,\beta}))_{k ,j=1}^N.$$
The number of unknowns is $N$ as $\mathbf{u}=u_N(\mathbf{x})$ with
$\mathbf{x}=[x_1^{\alpha,\beta},\cdots,x_N^{\alpha,\beta}]^T$.
\begin{remark}It is possible to implement the spectral collocation scheme for
(\ref{init-prob-RL-nonl-s4}) as:
to find $u_N\in E_N^{-\mu,\psi}$ such that
\begin{equation}\label{nonstand-spec-colloc1}
  {_\psi{\rm D}_{a,x}^{\mu}}u_N(x_j^{-\mu,\mu}) =f(x_j^{-\mu,\mu},u_N(x_j^{-\mu,\mu})),\quad j=1,2,\cdots,N
\end{equation}
and
$$ {_\psi{\rm D}_{a,x}^{-(1-\mu)}}u_N(x_0^{-\mu,\mu})=0.$$
In this circumstance, the mapped Lagrange interpolation basis functions should be replaced by
\begin{equation}\label{Lagrg_basis}
  L_j(x)=\frac{(\psi(x)-\psi_a)^{\mu}}{(\psi(x_j^{-\mu,\mu})-\psi_a)^{\mu}}\frac{\prod_{i\neq j}\psi(x)-\psi({x_i^{-\mu,\mu}})}
{\prod_{i\neq j}\psi(x_j^{-\mu,\mu})-\psi({x_i^{-\mu,\mu}})},
\quad j=0,\cdots,N.
\end{equation}
It is clear that the spectral collocation scheme (\ref{nonstand-spec-colloc1}) should achieve the same convergence rate as the Petrov-Galerkin scheme (\ref{Petrov-spec-Gal-ivp-rl}).
\end{remark}
\subsection{Fractional BVP}\label{subsec:4.2}
\subsubsection{Petrov-Galerkin spectral method}\label{subsec:4.2.1}
We consider the linear BVP (P3) with $f(x,u)=f(x)$, that is,
\begin{equation}\label{BVP-RL-P1-lin}
\left\{
  \begin{array}{ll}
    {_\psi{\rm D}^{\mu}_{a,x}} u(x)=f(x), \quad x\in(a,b), \quad 1<\mu<2, \\
    \left[{_\psi{\rm D}^{-(2-\mu)}_{a,x}} u\right](a)=\left[{_\psi{\rm D}^{-(2-\mu)}_{a,x}} u\right](b)=0.
  \end{array}
\right.
\end{equation}

Denote the discrete space as:
$$U_N:=\left\{\phi=(\psi(x)-\psi_a)^{\mu-1}\varphi:\varphi\in F_{N-1}^\psi\quad \mbox{such that}\quad {_\psi{\rm D}_{a,x}^{-(2-\mu)}}\phi(b)=0~\right\}.$$
Let $V_N^0:=F_{N}^{\psi}\cap H^1_{0,\psi}(a,b)$.
The Petrov-Galerkin spectral scheme of (\ref{BVP-RL-P1-lin}) is to find $u_N\in U_{N}$ such that $\forall w_N\in V_N^0$
\begin{equation}\label{spec-PGal-bvp-rl}
  ({_\psi{\rm D}_{a,x}^{\mu-1}}u_N, {\delta}^{1}_{\psi}w_N)_{\varpi^{0,0}}=-(f,
w_N)_{\varpi^{0,0}}.
\end{equation}
In the view of (\ref{weak-bvp-rlpsi}), the scheme (\ref{spec-PGal-bvp-rl}) can be equivalently reformulted as to find $v_N:={_\psi{\rm D}_{a,x}^{-(2-\mu)}}u_N\in V_N^0$ such that
\begin{equation}\label{spec-PGal-bvp-r2}
  ({\delta}^{1}_{\psi} v_N, {\delta}^{1}_{\psi}w_N)_{\varpi^{0,0}}=-(f,
w_N)_{\varpi^{0,0}}.
\end{equation}
\begin{theorem}Let $u$ and $u_N$ be the solution of (\ref{weak-bvp-rlpsi}) and (\ref{spec-PGal-bvp-r2}), respectively. If $v={_\psi{\rm D}_{a,x}^{-(2-\mu)}}u\in H^1_{0,\psi}(a,b)$ and ${\delta}^{1}_{\psi}v\in B^{m-1}_{0,0}(a,b)$, then we have
  \begin{equation}\label{err-est-bvp-g}
    \|{_\psi{\rm D}_{a,x}^{\mu-1}}(u-u_N)\|_{\varpi^{0,0}}\lesssim N^{1-m}\|
    {_\psi{\rm D}_{a,x}^{m+\mu-2}}u\|_{\varpi^{m-1,m-1}}.
  \end{equation}
Moreover, if $f\in B_{0,0}^{m-2}(a,b)(m\geq2)$, then we get
 \begin{equation}\label{err-est-bvp-f}
    \|{_\psi{\rm D}_{a,x}^{\mu-1}}(u-u_N)\|_{\varpi^{0,0}}\lesssim N^{1-m}\|
    {\delta}^{m-2}_{\psi}f\|_{\varpi^{m-1,m-1}}.
  \end{equation}
\end{theorem}
\begin{proof}~ Using the standard argument for error analysis of Galerkin approximation, we
know that
\begin{equation*}\label{est-galerkin-app}
  \|{\delta}^{1}_{\psi}(v-v_N)\|_{\varpi^{0,0}}=\inf_{v^*\in V_N^0}
  \|{\delta}^{1}_{\psi}(v-v^*)\|_{\varpi^{0,0}}.
\end{equation*}
By the approximation result of Theorem \ref{h1-proj-approx-res}, we have
\begin{equation*}\label{est-galerkin-app}
  \|{\delta}^{1}_{\psi}(v-v_N)\|_{\varpi^{0,0}}\leq C N^{1-m}\|{\delta}^{m}_{\psi} v\|
  _{\varpi^{m-1,m-1}}.
\end{equation*}
Noting that $v={_\psi{\rm D}_{a,x}^{-(2-\mu)}}u$ and $v_N={_\psi{\rm D}_{a,x}^{-(2-\mu)}}u_N$, we obtain the estimate (\ref{err-est-bvp-g}).

By ${\delta}^{m-2}_{\psi}f={_\psi{\rm D}_{a,x}^{m+\mu-2}}u$, the estimate (\ref{err-est-bvp-f}) follows immediately from (\ref{err-est-bvp-g}).
\end{proof}

For implementing the discrete scheme (\ref{spec-PGal-bvp-rl}), set
\begin{equation}\label{test-f-pgs}
  \phi_n(x)=(\psi(x)-\psi_a)^{\mu-1}J_{n,\psi}^{1-\mu,\mu-1}(x), \quad n=1,2,\cdots.
\end{equation}
Since $J_{n,\psi}^{-1,1}(x)=-\frac{\kappa(n+1)}{2n}(\psi_b-\psi(x))J_{n-1,\psi}^{1,1}(x)$, we have from Theorem \ref{thmIntMJFs}
\begin{equation*}
  \begin{split}
  {_\psi{\rm D}_{a,x}^{-(2-\mu)}} \phi_n(x)&=\frac{\Gamma(n+\mu)}{(n+1)!}(\psi(x)-\psi_a)J_{n,\psi}^{-1,1}(x)\\
  &=-\frac{\kappa\Gamma(n+\mu)}{2n n!}(\psi_b-\psi(x))(\psi(x)-\psi_a)J_{n-1,\psi}^{1,1}(x).
  \end{split}
\end{equation*}
 It verifies that $\{\phi_i\}_{i=1}^{N-1}$ consists a basis of $U_N$. Since
 $J_{n,\psi}^{-1,-1}(x)=-\frac{\kappa^2}{4}(\psi_b-\psi(x))(\psi(x)-\psi_a)J_{n-2,\psi}^{1,1}(x)$, then
 $$V_N^0=\mbox{span}\left\{J_{n,\psi}^{-1,-1}, \quad 2\leq n\leq N\right\}.$$

Let $u_N=\sum_{i=1}^{N-1}\widehat{u}_i \phi_i(x), w_N=J_{n,\psi}^{-1,-1}(n=2,\cdots,N)$, and insert them in (\ref{spec-PGal-bvp-rl}). Noting that $\delta_\psi^1J_{n,\psi}^{-1,-1}=\frac{\kappa(n-1)}{2}J_{n-1,\psi}^{0,0} (n\geq2)$, we have
\begin{equation}\label{sol-pg-bvp}
  \widehat{u}_i=-\frac{(2i+1) (i-1)!}{\kappa\Gamma(i+\mu)}(f,J_{i+1,\psi}^{-1,-1})_{\varpi^{0,0}},
\end{equation}
for $i=1,2,\cdots,N-1$. We note that using the MJFs as basis functions, the matrix of the linear system is diagonal.
\subsubsection{Spectral collocation method}\label{subsec:4.2.2}
Following the similar argument in Subsection \ref{subsubsec:4.1.2}, the
collocation method for boundary value problem needs to replace the fractional differential operator in equation by differentiation matrix corresponding to fractional differential operator with consideration of boundary value conditions. Consider abstract boundary value problem
\begin{equation}\label{ab-bvp}
  F({_\psi{\rm D}^{\mu_1}}u,{_\psi{\rm D}^{\mu_2}}u,\cdots, {_\psi{\rm D}^{\mu_k}}u, u,x)=0, \quad x\in(a,b)
\end{equation}
with $\mu_1>\mu_2\geq\cdots\geq\mu_k, \mu_1\in(1,2)$. The derivative operator ${_\psi{\rm D}^{\mu_k}}(1\leq i\leq k)$ can be ${_\psi{\rm D}_{a,x}^{\mu}},{_\psi{\rm D}_{x,b}^{\mu}}, {_{C\psi}{\rm D}_{a,x}^{\mu}},$ or
${_{C\psi}{\rm D}_{x,b}^{\mu}}$.
The boundary value conditions are given as $\mathcal{B}_au(a)=\mathcal{B}_bu(b)=0$, where
$\mathcal{B}_a$ and $\mathcal{B}_b$ are certain kinds of boundary value operators.
Using the differentiation matrix, the discrete scheme of (\ref{ab-bvp}) takes the following form
\begin{equation}\label{ab-bvp-dictret-M}
  F(\mathbf{D}^{\mu_1}\mathbf{u},\mathbf{D}^{\mu_2}\mathbf{u},\cdots, \mathbf{D}^{\mu_k}\mathbf{u}, \mathbf{u},\mathbf{x})=0,
\end{equation}
which will be solved for unknowns $\mathbf{u}$.
\section{Numerical validation} \label{sec:5}
We are now in the position of numerical tests. In the following numerical simulations, we consider various choices of $\psi$ as in Remark \ref{Remark_psi}. The main purpose of the examples is to check the convergence behaviour of the numerical solution. To measure the accuracy of the presented method if the exact solution is known, the errors are defined by
$$err(N)=\max\left\{|u_N(x_i)-u(x_i)|\right\},$$
where $x_i=a+\frac{(b-a)i}{500} (i=0,\cdots,500)$ for the Petrov-Galerkin method and $x_i=x_i^{\alpha,\beta} (i=0,\cdots,N)$ for the spectral collocation method,
$u_N(x)$ and $u(x)$ are the numerical and the exact solutions, respectively.
All of the computations are performed by Matlab R2020a on laptop
with AMD Ryzen 7 5800H with Radeon Graphics, 3.20 GHz.

\subsection{Fractional IVPs}
Let $0<\mu<1$. In this section, we apply the MJFs spectral and collocation method to the IVPs of linear and nonlinear fractional ODEs.
\begin{example}\label{examp-ivp-linear}Consider the linear FODE (\ref{init-prob-RL}) with
$0<\mu<1$. We choose $f(x)$ such that the
problem satisfies one of the following four cases:
\begin{enumerate}
  \setlength{\itemindent}{1.5em}
  \item [C11.]  $u(x)=(\psi(x)-\psi_a)^\mu \exp({\psi(x)-\psi_a}).$
  \item [C12.]  $u(x)=(\psi(x)-\psi_a)^\mu\sin(\pi\kappa(\psi(x)-\psi_a)).$
  \item [C13.]  $u(x)=(\psi(x)-\psi_a) \exp({\psi(x)-\psi_a}).$
  \item [C14.]  $u(x)=(\psi(x)-\psi_a)\sin(\pi\kappa(\psi(x)-\psi_a)).$
\end{enumerate}
\end{example}
We first apply the Petrov-Galerkin scheme (\ref{Petrov-spec-Gal-ivp-rl}) to solve the problem in this example. The $u(x)$ in C11 and C12 leads to a smooth $f(x)$, but
in C13 and C14 causes a low regularity of $f(x)$. According to Theorem \ref{convrate-PGS-ivp}, the predicted convergence rate is ``spectral accuracy" for C11 and C12, while it is limited for C13 and C14. The numerical errors of our tests are plotted in Figures \ref{fige1_1}-\ref{fige1_6}. The ``spectral accuracy" is observed from Figures \ref{fige1_1}-\ref{fige1_3}, whilst the convergence rate is finite in
Figures \ref{fige1_4}-\ref{fige1_6}, which agree with the theoretical results.
\begin{figure}[t]
	\centering
		\centerline{\includegraphics[]{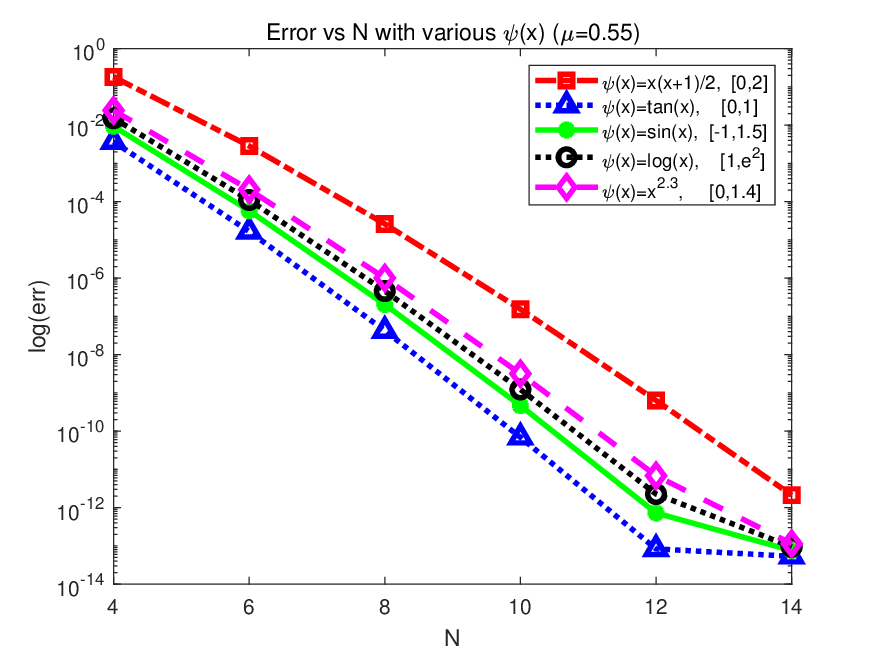}}
		{\caption{Errors of Example \ref{examp-ivp-linear} by Petrov-Galerkin scheme for case C11 with different $\psi(x)$.}\label{fige1_1}}
\end{figure}

\begin{figure}[t]
	\centering	
		\centerline{\includegraphics[]{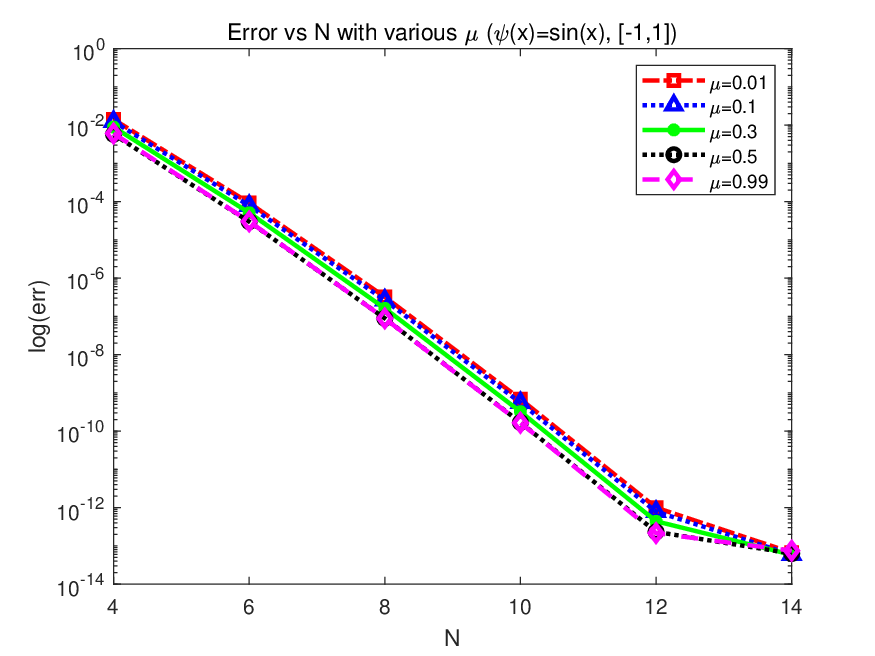}}
		{\caption{Errors of Example \ref{examp-ivp-linear} by Petrov-Galerkin scheme for case C11 with different $\mu$.}\label{fige1_2}}
\end{figure}

\begin{figure}[t]
	\centering
		\centerline{\includegraphics[]{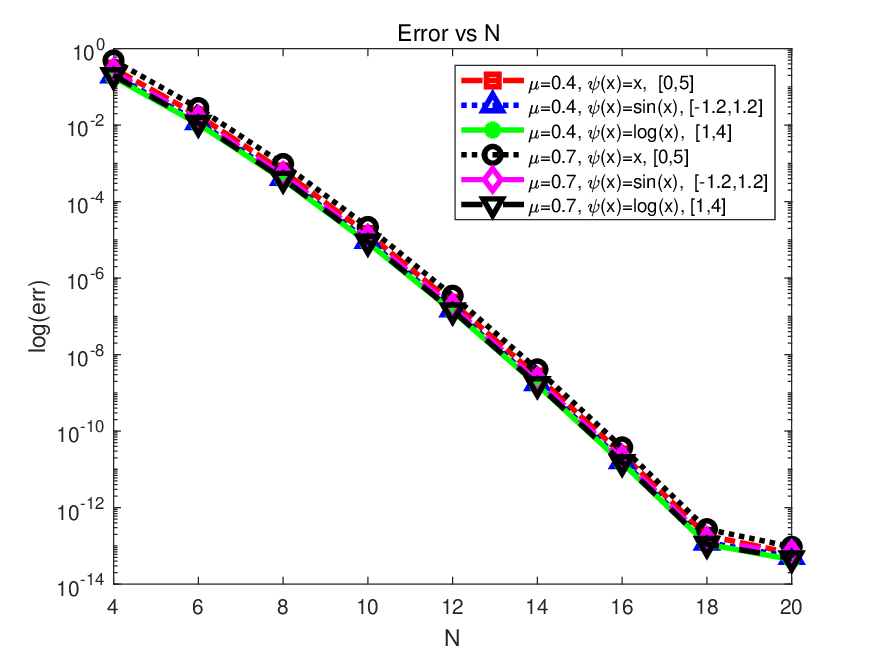}}
		{\caption{Errors of Example \ref{examp-ivp-linear} by Petrov-Galerkin scheme for case C12.}\label{fige1_3}}
\end{figure}

\begin{figure}[t]
	\centering
		\centerline{\includegraphics[]{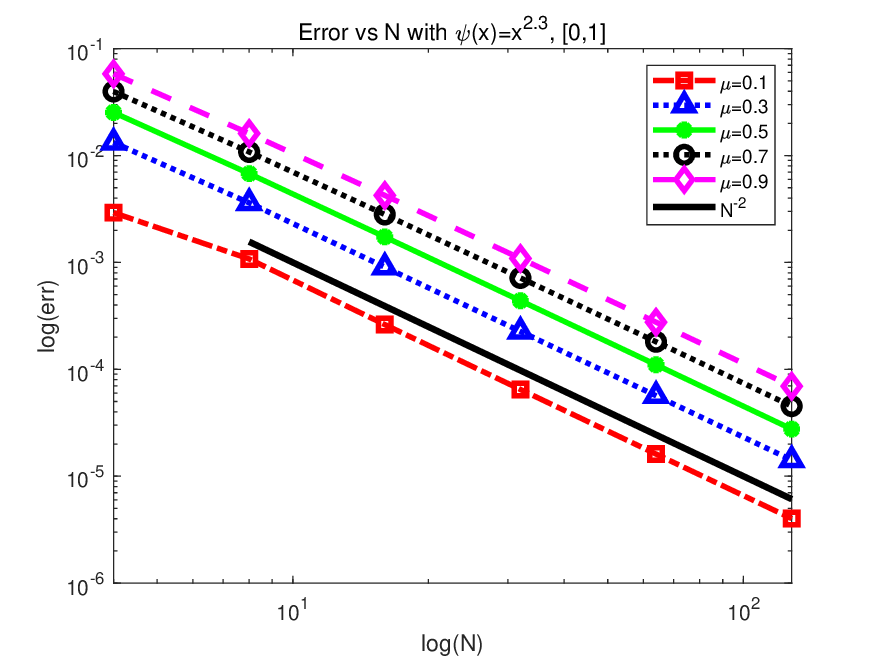}}
		{\caption{Errors of Example \ref{examp-ivp-linear} by Petrov-Galerkin scheme for case C13.}\label{fige1_4}}
\end{figure}

\begin{figure}[t]
	\centering
		\centerline{\includegraphics[]{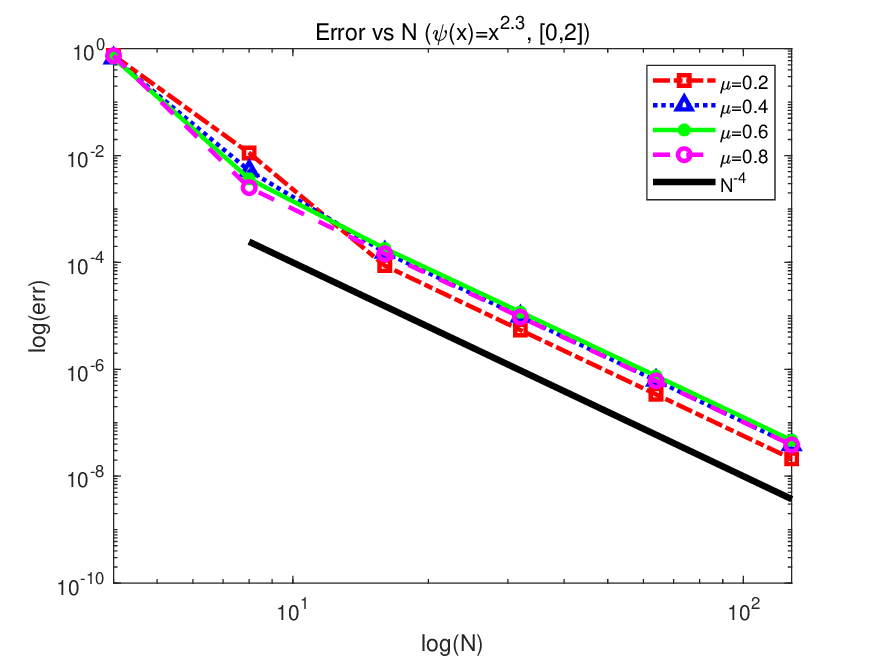}}
		{\caption{Errors of Example \ref{examp-ivp-linear} by Petrov-Galerkin scheme for case C14.}\label{fige1_5}}
\end{figure}

\begin{figure}[t]
\centering
		\centerline{\includegraphics[]{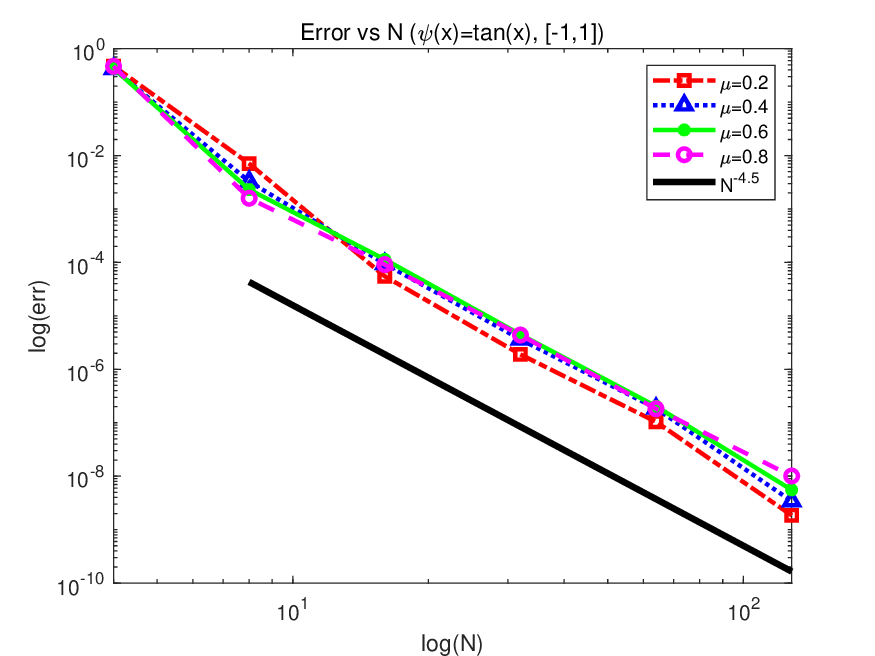}}
		{\caption{Errors of Example \ref{examp-ivp-linear} by Petrov-Galerkin scheme for case C14.}\label{fige1_6}}
\end{figure}

We next apply the spectral collocation scheme (\ref{init-prob-RL-dis}) to solve the same problem for comparison. The parameters $\alpha$ and $\beta$ in the scheme (\ref{init-prob-RL-dis}) are always zero for simplification. The numerical errors of our tests are plotted in Figures \ref{fige1_c1}-\ref{fige1_c6}. The convergence behaviour of the spectral collocation scheme is pretty interesting: the convergence rate is finite in Figures \ref{fige1_c1}-\ref{fige1_c3} for the low regularity of the exact solution $u(x)$ (C11 and C12), whilst the ``spectral accuracy" is achieved for the smooth exact solution $u(x)$ (C13 and C14) in Figures \ref{fige1_c4}-\ref{fige1_c6}.

\begin{figure}[t]
	\centering
		\centerline{\includegraphics[]{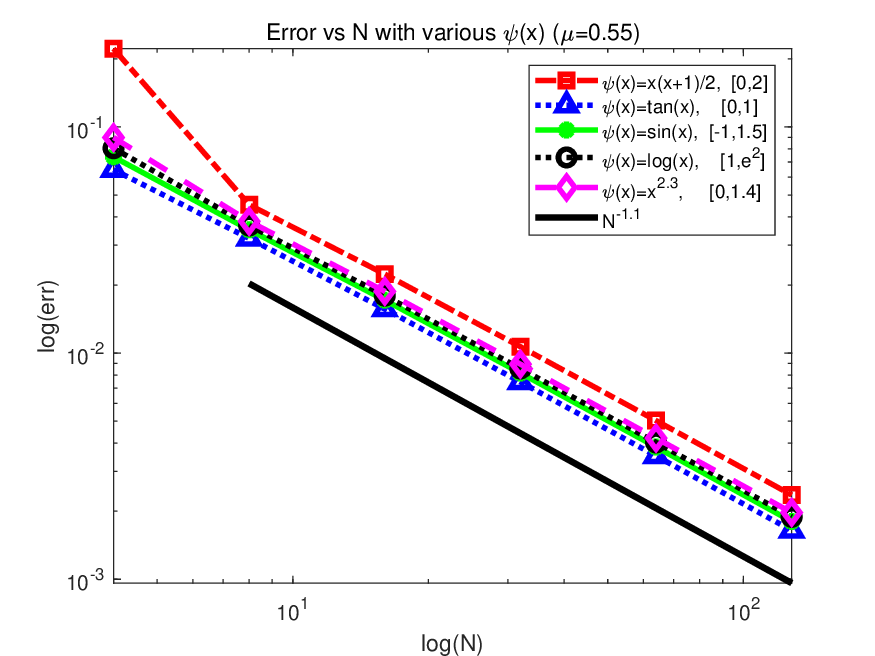}}
		{\caption{Errors of Example \ref{examp-ivp-linear} by spectral collocation scheme for case C11 with different $\psi(x)$.}\label{fige1_c1}}
\end{figure}

\begin{figure}[t]
\centering
		\centerline{\includegraphics[]{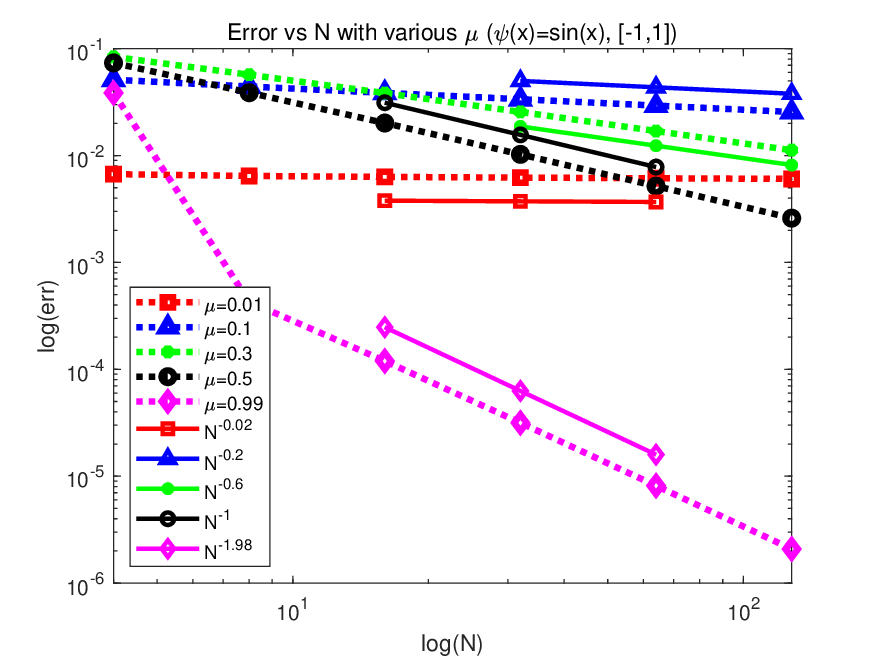}}
		{\caption{Errors of Example \ref{examp-ivp-linear} by spectral collocation scheme for case C11 with different $\mu$.}\label{fige1_c2}}
\end{figure}

\begin{figure}[t]
	\centering
		\centerline{\includegraphics[]{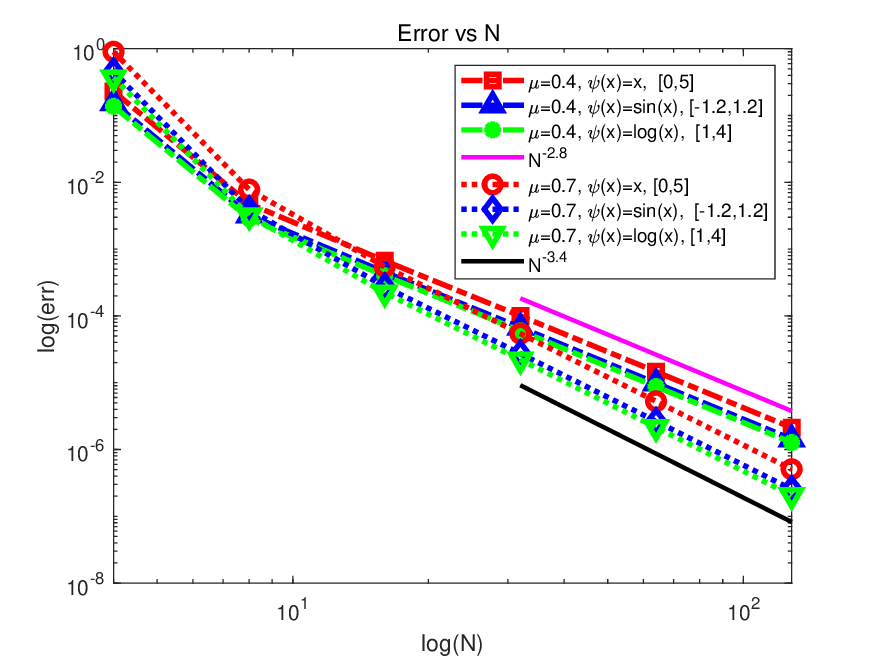}}
		{\caption{Errors of Example \ref{examp-ivp-linear} by spectral collocation scheme for case C12.}\label{fige1_c3}}
\end{figure}

\begin{figure}[t]
\centering
		\centerline{\includegraphics[]{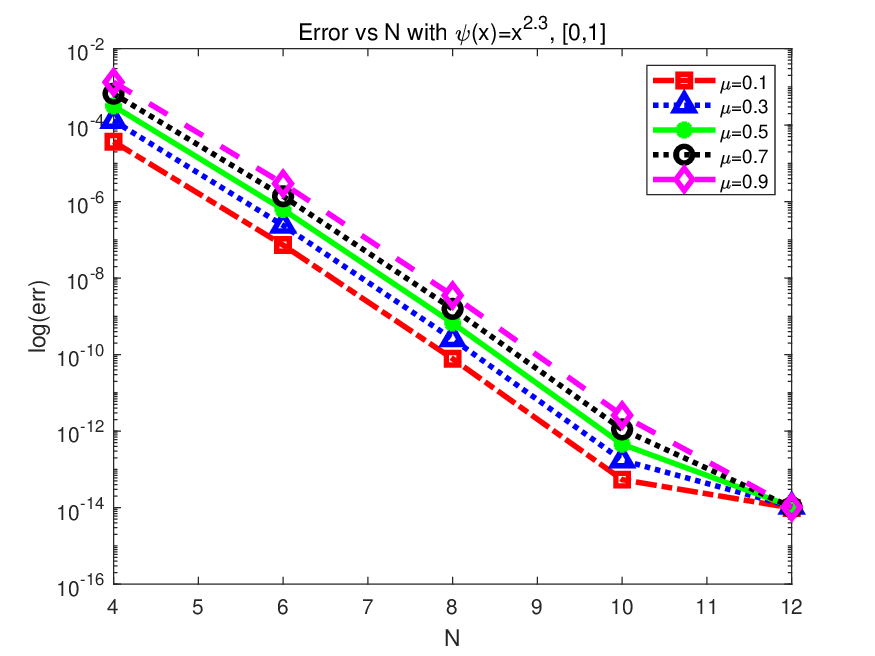}}
		{\caption{Errors of Example \ref{examp-ivp-linear} by spectral collocation scheme for case C13.}\label{fige1_c4}}
\end{figure}

\begin{figure}[t]
	\centering
		\centerline{\includegraphics[]{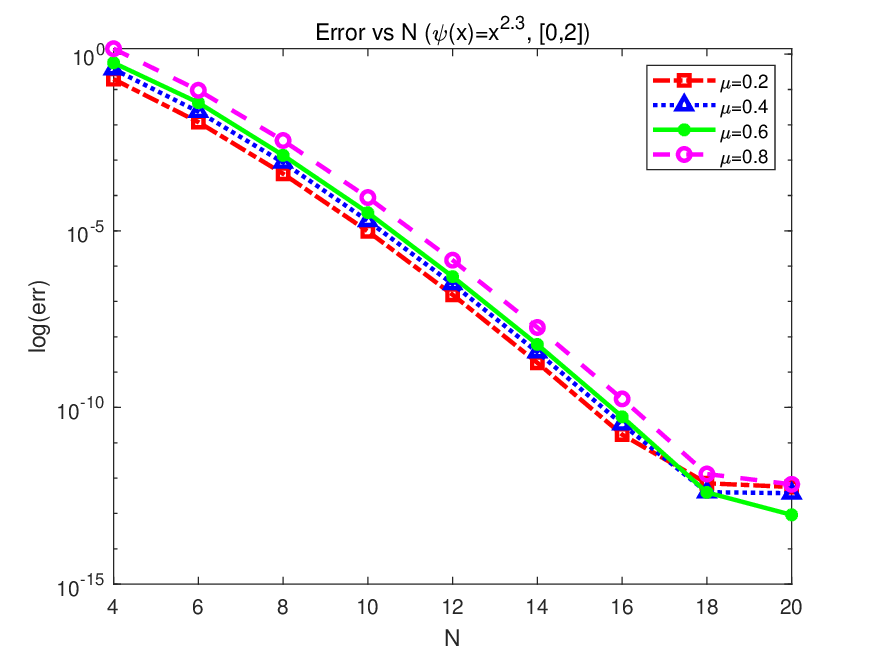}}
		{\caption{Errors of Example \ref{examp-ivp-linear} by spectral collocation scheme for case C14.}\label{fige1_c5}}
	\end{figure}

\begin{figure}[t]
\centering
		\centerline{\includegraphics[]{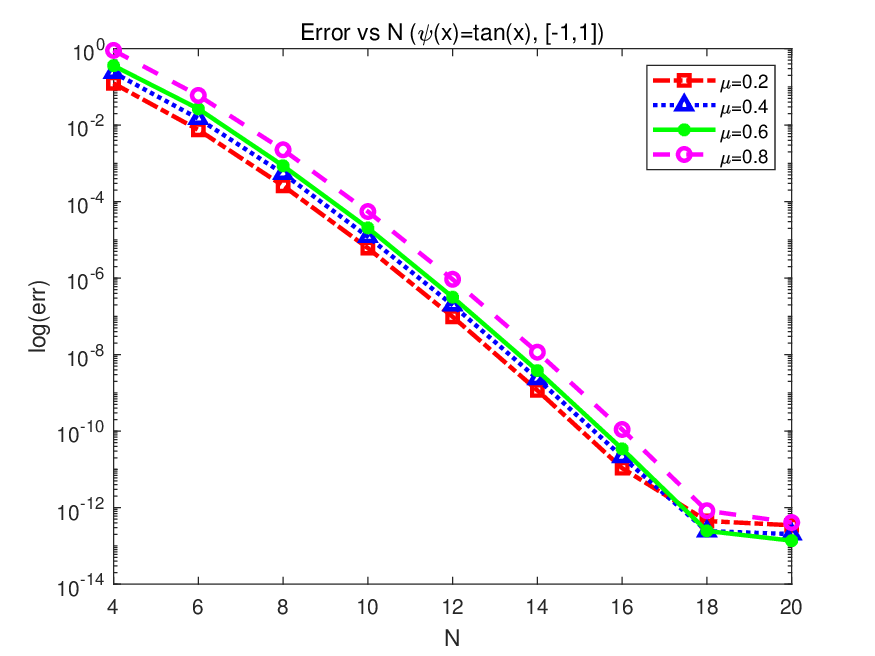}}
		{\caption{Errors of Example \ref{examp-ivp-linear} by spectral collocation scheme for case C14.}\label{fige1_c6}}
\end{figure}

\begin{example}\label{examp-ivp-nonlinear}For $0<\mu<1$, consider
\begin{equation}\label{examp-nonlin-ivp}
  {_{C\psi}{\rm D}_{a,x}^{\mu}}u(x)=g(x)-u^2,\quad x\in(a,b],\quad u(a)=0,
\end{equation}
where $g(x)$ is chosen such that the
problem satisfies the following two cases:
\begin{enumerate}
\setlength{\itemindent}{1.5em}
  \item [C21.]  $u(x)=\sum_{k=1}^{30}\frac{\Gamma(k-\mu+1)(\psi(x)-\psi_a)^k}{k!}.$
  \item [C22.]  $u(x)=\frac{120(\psi(x)-\psi_a)^{5+\mu}}{\Gamma(\mu+6)}+
      \frac{\Gamma(\mu+4)(\psi(x)-\psi_a)^{3+2\mu}}{\Gamma(2\mu+4)}
      +\frac{\Gamma(2\mu+3)(\psi(x)-\psi_a)^{2+3\mu}}{\Gamma(3\mu+3)}.$
\end{enumerate}
\end{example}
The Petrov-Galerkin spectral scheme is awkward to solve the above problem because of nonlinearity. Here we apply spectral collocation scheme (\ref{init-prob-RL-dis}) to solve it. The exact solution $u\in F_{30}^\psi$ in C21, so we take $N = 30$ and apply the Newton method to solve the nonlinear system with total $10$ iterations. We list the numerical errors for various $\psi(x)$ and $\mu$ in Table \ref{tabe2_1}. For the C22, we only list the numerical errors with $\psi(x)=x^{2.3}, x\in [0,1]$ and different $\mu$ in Table \ref{tabe2_2}. For other $\psi(x)$, the numerical results are similar.

\begin{table}[h]
\caption{Errors of Example \ref{examp-ivp-nonlinear} by spectral collocation scheme with $N=30$ in the case C21 with different $\mu$ and $\psi_1=\frac{x(x+1)}{2},[0,0.7];\psi_2=\tan(x),[0,\frac{\pi}{5}]; \psi_3=\sin(x),[0,\frac{\pi}{3}];\psi_4=\log(x),[1,e];
\psi_5=x^{2.3},[0,1].$}\label{tabe2_1}
\begin{tabular}{c|ccccc}\hline
$\mu$ &  0.1 &  0.3  & 0.5 & 0.7  & 0.9 \\ \hline
$\psi_1$& 6.239e-14 &9.304e-14 &9.481e-14 &9.048e-14 &9.193e-14 \\
$\psi_2$& 7.327e-14 &1.315e-13 &1.410e-13 &1.308e-13 &1.259e-13 \\
$\psi_3$& 8.260e-14 &1.954e-13 &2.411e-13 &2.087e-13 &1.861e-13 \\
$\psi_4$& 8.527e-14 &2.700e-13 &5.169e-13 &4.761e-13 &3.597e-13 \\
$\psi_5$& 8.527e-14 &2.718e-13 &5.205e-13 &4.778e-13 &3.606e-13 \\
\hline
\end{tabular}
\end{table}

\begin{table}[h]
\caption{Errors of Example \ref{examp-ivp-nonlinear} by spectral collocation scheme in the case C22 of $\psi=x^{2.3},[0,1]$ with different $\mu$ and $N$.}\label{tabe2_2}
\begin{tabular}{c|ccccc}\hline
$N$ &  $\mu=0.1$ &  $\mu=0.3$  & $\mu=0.5$ & $\mu=0.7$  & $\mu=0.9$  \\ \hline
 4 &4.321e-04 &1.216e-03 &3.847e-03 &1.130e-02 &2.641e-02 \\
 8 &1.890e-06 &1.438e-06 &1.732e-07 &1.453e-07 &4.692e-07 \\
16 &6.983e-08 &1.049e-08 &5.052e-10 &9.476e-10 &3.269e-10 \\
32 &2.379e-09 &8.250e-11 &1.104e-12 &2.439e-12 &4.344e-13 \\
64 &7.771e-11 &6.462e-13 &1.699e-13 &1.095e-13 &7.183e-14 \\
\hline
\end{tabular}
\end{table}

\subsection{Fractional BVPs}
Let $1<\mu<2$. In this subsection, we apply the MJFs spectral and collocation method to the BVPs with the $\psi$-fractional derivative. As a standard example, we
consider $\psi$-fractional Helmholtz equation
\begin{equation}\label{bvp-helmh}
  \lambda^2 u(x) - {_{\psi}{\rm D}_{a,x}^{\mu}}u(x)=f(x), \quad x\in (a,b).
\end{equation}
The homogenous Dirichlet condition  $u(a)=u(b)=0$ is used here in order to ensure the wellposedness of the above model problem.

In the Petrov-Galerkin scheme (\ref{spec-PGal-bvp-rl}), denoting the mass and stiff matrix by $\mathbf{M}=(m_{ij})$ and $\mathbf{S}=(s_{ij})$ with
\begin{eqnarray*}
  m_{ij} &=& (\phi_i,J_{j+1,\psi}^{-1,-1})_{\varpi^{0,0}},\quad i,j=1,\cdots,N-1, \\
  s_{ij} &=& ({_{\psi}^R{\rm D}_{a,x}^{\mu-1}}\phi_i,{\delta }^{1}_{\psi}J_{j+1,\psi}^{-1,-1})_{\varpi^{0,0}}, \quad i,j=1,\cdots,N-1,
\end{eqnarray*}
we have the discrete representation
$$\left(\lambda^2 \mathbf{M} +\mathbf{S}\right)\mathbf{u}=\mathbf{f},$$
where $\mathbf{f}=[(f,J_{2,\psi}^{-1,-1})_{\varpi^{0,0}},\cdots,(f,J_{N,\psi}^{-1,-1})_{\varpi^{0,0}}]^T$ and $\mathbf{u}=
[\widehat{u}_1,\cdots,\widehat{u}_{N-1}]^T$.
As stated in Subsection \ref{subsec:4.2.1}, the stiff matrix is diagonal with $s_{ii}=\frac{\kappa\Gamma(i+\mu)}{(2i+1)\Gamma(i)}$. However, the mass matrix is full, which can be computed exactly by the Gauss-MJF-type quadrature in Theorem \ref{Gauss-MJF-quad}.

In the spectral collocation scheme (\ref{ab-bvp-dictret-M}), making use of the differentiation matrix $\mathbf{D}^\mu$, we have the discrete representation
$$\left(\lambda^2 \mathbf{I} -\mathbf{D}^\mu\right)\mathbf{u}=\mathbf{f},$$
where $\mathbf{I}$ is the identity matrix, $\mathbf{u}=u(\mathbf{x}), \mathbf{f}=
f(\mathbf{x})$ and $\mathbf{x}=[x_1^{\alpha,\beta},\cdots,x_{N-1}^{\alpha,\beta}]^T$.

\begin{example}\label{examp-bvp}Consider the following three cases:
\begin{enumerate}
  \setlength{\itemindent}{1.5em}
  \item [C31.] Solution: $u(x)=(\psi(x)-\psi_a)^{\mu-1}\sum_{k=1}^{30}\frac{J_{k,\psi}^{1-\mu,\mu-1}(x)}{k!}.$
  \item [C32.] Smooth solution: $u(x)=\sin(\kappa\pi(\psi(x)-\psi_a)).$
  \item [C33.] Solution of low regularity: $u(x)=(\psi_b-\psi(x))(\psi(x)-\psi_a)^{2\mu}.$
\end{enumerate}
\end{example}

\begin{table}[h]
\caption{Errors of Example \ref{examp-bvp} ($\lambda=0$) by Petrov-Galerkin scheme in the case C31 of $\psi=x(x+1)/2,[1,3]$ with different $\mu$ and $N$.}\label{tabe31_1}
\begin{tabular}{c|ccccc}\hline
$N$ &  $\mu=1.1$ &  $\mu=1.3$  & $\mu=1.5$ & $\mu=1.7$  & $\mu=1.9$  \\ \hline
  4 &4.863e-02 &4.088e-02 &3.086e-02 &3.289e-02 &4.054e-02 \\
 8 &2.470e-05 &1.833e-05 &1.215e-05 &1.357e-05 &1.640e-05 \\
12 &1.930e-09 &1.325e-09 &8.122e-10 &9.226e-10 &1.115e-09 \\
16 &4.952e-14 &3.331e-14 &2.176e-14 &2.176e-14 &2.709e-14 \\
18 &5.773e-15 &5.995e-15 &5.551e-15 &5.329e-15 &7.994e-15 \\
\hline
\end{tabular}
\end{table}

It verifies that $u(a)=0,u(b)\neq0$ and
${_\psi{\rm D}_{a,x}^{-(2-\mu)}}u(a)={_\psi{\rm D}_{a,x}^{-(2-\mu)}}u(b)=0$ in the case C31. Hence the Petrov-Galerkin scheme is used to solve the problem (\ref{bvp-helmh}) with $\lambda=0$ for the case C31. The numerical errors are listed in Table \ref{tabe31_1}. It is shown that the spectral accuracy is achieved for various fractional order $\mu$.

\begin{table}[h]
\caption{Errors of Example \ref{examp-bvp} ($\lambda=0$) by spectral collocation scheme ($\alpha=\beta=0$) in the case C32 of $\psi=x(x+1)/2,[1,4]$ with different $\mu$ and $N$.}\label{tabe32_1}
\begin{tabular}{c|ccccc}\hline
$N$ &  $\mu=1.1$ &  $\mu=1.3$  & $\mu=1.5$ & $\mu=1.7$  & $\mu=1.9$  \\ \hline
 4 &1.836e+00 &7.209e-01 &3.675e-01 &1.384e-01 &8.632e-02 \\
 8 &4.660e-03 &1.852e-03 &9.221e-04 &3.823e-04 &1.096e-04 \\
12 &1.893e-06 &7.428e-07 &3.508e-07 &1.379e-07 &3.547e-08 \\
16 &2.223e-10 &8.576e-11 &3.870e-11 &1.481e-11 &3.727e-12 \\
20 &8.604e-14 &5.862e-14 &2.387e-14 &7.894e-14 &1.499e-14 \\
\hline
\end{tabular}
\end{table}

\begin{table}[h]
\caption{Errors of Example \ref{examp-bvp} ($\mu=1.55$) by spectral collocation scheme ($\alpha=\beta=0.5$) in the case C32 of $\psi=\tan(x),[0,\pi/3]$ with different $\lambda$ and $N$.}\label{tabe32_2}
\begin{tabular}{c|ccccc}\hline
$N$ &  $\lambda=1$ &  $\lambda=10$  & $\lambda=10^2$ & $\lambda=10^3$  & $\lambda=10^4$  \\ \hline
 4 &3.744e-01 &1.648e-02 &1.677e-04 &1.677e-06 &1.677e-08 \\
 8 &1.444e-03 &2.042e-04 &2.146e-06 &2.146e-08 &2.146e-10 \\
12 &6.794e-07 &2.020e-07 &2.468e-09 &2.468e-11 &2.468e-13 \\
16 &8.621e-11 &3.872e-11 &6.170e-13 &6.162e-15 &4.441e-16 \\
20 &4.796e-14 &8.923e-15 &6.523e-16 &3.331e-16 &2.220e-16 \\
\hline
\end{tabular}
\end{table}

\begin{table}[h]
\caption{Errors of Example \ref{examp-bvp} ($\mu=1.45$) by spectral collocation scheme ($\lambda=10$) in the case C32 of $\psi=\sin(x),[-\pi/3,\pi/3]$ with different $\alpha,\beta$ and $N$.}\label{tabe32_3}
\begin{tabular}{c|ccccc}\hline
$N$ &  $\alpha=\beta=-0.5$ &  $\alpha=\beta=0.5$  & $\alpha=\beta=1$ & $\alpha=0.5,\beta=-0.5$  & $\alpha=-0.5,\beta=2$  \\ \hline
4 &1.991e-02 &1.610e-02 &1.723e-02 &2.380e-02 &7.338e-03 \\
 8 &1.102e-04 &1.852e-04 &2.222e-04 &3.113e-04 &1.122e-04 \\
12 &5.306e-08 &1.933e-07 &2.807e-07 &3.483e-07 &1.072e-07 \\
16 &7.560e-12 &4.059e-11 &7.115e-11 &7.624e-11 &1.897e-11 \\
20 &1.077e-14 &7.813e-15 &3.225e-14 &1.979e-14 &8.632e-15 \\
\hline
\end{tabular}
\end{table}

We use the spectral collocation scheme to solve the problem (\ref{bvp-helmh}) for the cases C32 and C33 by taking different values of $N,\alpha,\beta,\mu,\lambda$.  The numerical errors are listed in Tables \ref{tabe32_1}-\ref{tabe33_1}. It is shown that the spectral accuracy is achieved in the case C32 for different values of parameters since the solution is smooth. However, the finite convergent order is observed from Table \ref{tabe33_1} for the case C33 corresponding to the low regularity of the solution.

\begin{table}[h]
\caption{Errors of Example \ref{examp-bvp} ($\alpha=\beta=-0.5$) by spectral collocation scheme ($\lambda=2$) in the case C33 of $\psi=\sin(x),[-\pi/3,\pi/3]$ with different $\mu$ and $N$.}\label{tabe33_1}
\begin{tabular}{c|ccccc}\hline
$N$ &  $\mu=1.02$ &  $\mu=1.2$  & $\mu=1.4$ & $\mu=1.6$  & $\mu=1.8$  \\ \hline
4 &9.669e-04 &6.346e-03 &3.984e-03 &5.256e-03 &1.712e-02 \\
 8 &8.021e-05 &2.006e-04 &4.521e-05 &2.098e-05 &1.692e-05 \\
16 &6.623e-06 &7.124e-06 &7.984e-07 &2.051e-07 &8.867e-08 \\
32 &4.620e-07 &2.514e-07 &1.611e-08 &2.332e-09 &5.695e-10 \\
64 &2.931e-08 &8.960e-09 &3.304e-10 &2.735e-11 &3.819e-12 \\
128 &1.783e-09 &3.210e-10 &6.809e-12 &3.215e-13 &1.429e-13 \\
256 &1.058e-10 &1.152e-11 &1.659e-13 &1.751e-13 &1.537e-13 \\
\hline
\end{tabular}
\end{table}

\begin{remark} The interesting facts about the convergence behaviour of the spectral method and collocation method are confirmed by the numerical tests, which are summarized below:
\begin{enumerate}
  \item [$\bullet$]The exponential convergence is demonstrated for both the Petrov-Galerkin method and the spectral collocation method but in the different settings.
  \item [$\bullet$]The function $\psi(x)$ has little effect on the convergence behaviour of our numerical method. The differentiation matrices of the spectral collocation for various options of $\psi(x)$ are almost the same.
  \item [$\bullet$]The well-designed Petrov-Galerkin spectral method can achieve both the exponential convergence and the well-conditioned discrete system. The spectral collocation method can achieve the same convergence rate as the Petrov-Galerkin spectral method.
 \item [$\bullet$]The influence of parameters $\alpha,\beta$ of the spectral collocation scheme upon the computational accuracy is not obvious.
\end{enumerate}
\end{remark}

\section{Conclusions} \label{sec:6}
The main reason to consider the $\psi$-fractional differential equations lies in: (i) We have a more general approach to solve various problems by the choices of $\psi(x)$. (ii)
We can discover the inner structure of various fractional order differential equations through numerical simulation. (iii) We can use these tools to explain and predict natural phenomena in the real world.

In this paper, we deal with the spectral approximation to the fractional calculus with respect to function $\psi$, also named as $\psi$-fractional calculus, which generalizes the Hadamard and the common-used Riemann-Liouville fractional calculi. We consider spectral-type methods with mapped Jacobi functions as basis functions and obtain efficient algorithms to solve $\psi$-fractional differential equations. In particular, we setup the Petrov-Galerkin  spectral method and spectral collocation method for initial and boundary value problems involving Rieman-Liouville and Caputo $\psi$-fractional derivatives. We develop basic approximation theory for the MJFs and conduct the error estimates of the presented methods. We also derive a recurrence relation to evaluate the collocation differentiation matrix for implementing the spectral collocation algorithm. Numerical examples confirm the theoretical results and demonstrate the effectiveness of the spectral and collocation methods.

The demand of the flexibility of the numerical algorithm pushes the multi-domain method into consideration \cite{ZhaoMK19}. The presented method can be designed into a multi-domain version. Besides, we note that it is not difficult to employ the spectral collocation method to solve the variable-order fractional differential equations with more general form \cite{ZhaZ23}. We will conduct the error analysis of the method for solving the variable-order (non-constant case) $\psi$-fractional differential problem in our further work. Meanwhile, since the well-posedness and regularity
of the variable-order fractional differential equations are still open \cite{ZheW22}, we also expect to discuss them in the future plan.

\section{Acknowledgement}
 C. L. was partly supported by the National Natural Science Foundation of China under Grant No. 12271339.

\section*{Appendix A}
We first give the generalized Gronwall's inequality.

\noindent{\textbf{Theorem A.1.}}\emph{
 Assume that $\xi(x)$ is a nonnegative function locally integrable on $[a,b]$, $C>0$ is a constant, and $\eta(x)$ is continuous, non-negative such that
$$ \eta(x)\leq \xi(x)+C\int_a^x\psi'(s)(\psi(x)-\psi(s))^{\mu-1}\eta(s){\rm d}s,\quad x\in[a,b], $$
Then for $x\in[a,b]$, there holds
\begin{equation}\label{G-Gronwallinq1}
\eta(x)\leq \xi(x)+ \int_{a}^{x}\sum_{n=1}^{\infty}\frac{(C\Gamma(\mu))^n}{\Gamma(n\mu)}
\psi'(\tau)(\psi(x)-\psi(\tau))^{n\mu-1}\xi(\tau){\rm d}\tau.
\end{equation}
In particular, if $\xi(x)$ is be a nondecreasing function on $[a,b]$, then
\begin{equation}\label{G-Gronwallinq-2}
\eta(x)\leq\xi(x)E_{\mu}\left(C\Gamma(\mu)(\psi(x)-\psi_a)^\mu\right),
\end{equation}
where $E_{\beta}(s):=\sum_{k=0}^\infty\frac{s^k}{\Gamma(1+k\beta)}$ denotes the Mittag-Leffler function.}
\begin{proof}~ The theorem is a special case $\lambda=0$ of Theorem 2.3 in \cite{KucMFF22}.
\end{proof}

We estimate the projection error at the endpoints as follows.

\noindent{\textbf{Theorem A.2.}} \emph{
Let $\alpha,\beta>-1$ and $u\in B^m_{\alpha,\beta}(a,b)$, we have\begin{enumerate}
          \item If $\alpha+1<m\leq N+1$, then
\begin{equation*}
  |(u-\mathbf{P}_N^{\alpha,\beta}u)(b)|\leq cm^{-1/2}N^{1+\alpha-m}\|{\delta}^{m}_{\psi} u\|_{\varpi^{\alpha+m,\beta+m}};
\end{equation*}
          \item If $\beta+1<m\leq N+1$, then
\begin{equation*}
  |(u-\mathbf{P}_N^{\alpha,\beta}u)(a)|\leq cm^{-1/2}N^{1+\beta-m}\|{\delta}^{m}_{\psi} u\|_{\varpi^{\alpha+m,\beta+m}},
\end{equation*}
        \end{enumerate}
where $c$ is a positive constant independent of $m,N$ and $u$.}
\begin{proof}~ For $n\geq k$, denote $h_{n,k}^{\alpha,\beta}=(d_{n,k}^{\alpha,\beta})^2\gamma_{n-k}^{\alpha+k,\beta+k}$ and let $\widetilde{m}=\min\{m,N+1\},$ by Cauchy-Schwarz inequality we have
\begin{equation*}
  \begin{split}
&|(u-\mathbf{P}_N^{\alpha,\beta}u)(b)|=\left|\sum_{n=N+1}^\infty\widehat{u}_nJ_{n,\psi}^{\alpha,\beta}(b)\right|
 \leq\sum_{n=N+1}^\infty|\widehat{u}_n||J_{n,\psi}^{\alpha,\beta}(b)|\\
\leq&
\left(\sum_{n=N+1}^\infty|J_{n,\psi}^{\alpha,\beta}(b)|^2(h_{n,\widetilde{m}}^{\alpha,\beta})^{-1}\right)^{1/2}
\left(\sum_{n=N+1}^\infty|\widehat{u}_n|^2h_{n,\widetilde{m}}^{\alpha,\beta}\right)^{1/2}\\
\leq&
\left(\sum_{n=N+1}^\infty|J_{n,\psi}^{\alpha,\beta}(b)|^2(h_{n,\widetilde{m}}^{\alpha,\beta})^{-1}\right)^{1/2}
\|{\delta}^{\widetilde{m}}_{\psi}u\|_{\varpi^{\alpha+\widetilde{m},\beta+\widetilde{m}}}.
\end{split}
\end{equation*}
Since $J_{n,\psi}^{\alpha,\beta}(b)=P_n^{\alpha,\beta}(1)$, along the same arguments as Lemma 3.10 in \cite{SheTW11} (page 135), we have
\begin{equation*}
  \sum_{n=N+1}^\infty|J_{n,\psi}^{\alpha,\beta}(b)|^2(h_{n,\widetilde{m}}^{\alpha,\beta})^{-1}
\leq c \widetilde{m}^{-1}N^{2(1+\alpha-\widetilde{m})}.
\end{equation*}
This gives the first inequality. Since $J_{n,\psi}^{\alpha,\beta}(a)=P_n^{\alpha,\beta}(-1)$, we derive the second inequality as above.
\end{proof}

The stability of the interpolation operator is given in the following theorem.\\
\noindent{\textbf{Theorem A.3.}} \emph{For any $v(x)\in C[a,b]\cap B^1_{\alpha,\beta}(a,b)$, we have
\begin{equation*}
  \|\mathbf{I}_N^{\alpha,\beta}v\|_{\varpi^{\alpha,\beta}}\lesssim \left(N^{-\alpha-1}|v(b)|+N^{-\beta-1}|v(a)|+\|v\|_{\varpi^{\alpha,\beta}}+
\kappa^{-2} N^{-1}|v|_{B^1_{\alpha,\beta}}\right).
\end{equation*}}
\begin{proof}~ Note that $x(s)=\psi^{-1}\left(\frac{s+1}{\kappa}+\psi_a\right)$ and
let $\widetilde{v}(s)=v(x(s))$. Then
\begin{equation*}
  \mathbf{I}_N^{\alpha,\beta}v(x)=I_N^{\alpha,\beta}\widetilde{v}(s):=\sum_{j=0}^N
\widetilde{v}(s_j^{\alpha,\beta})l_j(s),\quad s\in[-1,1],
\end{equation*}
with $I_N^{\alpha,\beta}$ the Jacobi-Gauss-Lobatto interpolation operator and
$$l_j(s)=\prod_{i=0,~i\neq j}^N\frac{s-s_i^{\alpha,\beta}}{s_ j^{\alpha,\beta}-s_i^{\alpha,\beta}},\quad j=0,1,\cdots,N$$
the standard Lagrange interpolation basis functions at nodes $\{s_j^{\alpha,\beta}\}_{j=0}^N$.

Because of
\begin{equation*}
  \int_a^b \left(v(x)\right)^2\varpi^{\alpha,\beta}(x){\rm d}x=\int_{-1}^1
\left(\widetilde{v}(s)\right)^2\omega^{\alpha,\beta}(s){\rm d}s,
\end{equation*}
this means that
$\|v\|_{\varpi^{\alpha,\beta}}=\|\widetilde{v}\|_{\omega^{\alpha,\beta}}$.
Moreover, by
\begin{equation*}
  \int_{-1}^1\left(\frac{{\rm d}\widetilde{v}}{{\rm d}s}\right)^2\omega^{\alpha,\beta}(s){\rm d}s=
\int_a^b\left({\kappa^{-1}}{{\delta}^{1}_{\psi}v}\right)^2\varpi^{\alpha,\beta}(x){\rm d}x=\frac{1}{\kappa^2}\left\|{\delta}^{1}_{\psi}v\right\|_{\varpi^{\alpha,\beta}},
\end{equation*}
we have
$\|{\delta}^{1}_{\psi} v\|_{\varpi^{\alpha,\beta}}={\kappa^{2}}\|\partial_s\widetilde{v}\|_{\omega^{\alpha,\beta}}$.

Now it follows from Lemma 3.11 in \cite{SheTW11} that
\begin{equation*}
\begin{split}
  \|I_N^{\alpha,\beta}\widetilde{v}(s)\|_{\omega^{\alpha,\beta}}&\lesssim
N^{-\alpha-1}|\widetilde{v}(1)|+N^{-\beta-1}|\widetilde{v}(-1)|+
\|\widetilde{v}\|_{\omega^{\alpha,\beta}}+
N^{-1}|\widetilde{v}|_{1,\omega^{\alpha+1,\beta+1}}\\
&\lesssim
N^{-\alpha-1}|{v}(b)|+N^{-\beta-1}|{v}(a)|+
\|{v}\|_{\varpi^{\alpha,\beta}}+\kappa^{-2}
N^{-1}|{v}|_{B^1_{\alpha,\beta}}.
\end{split}
\end{equation*}
This ends the proof.
\end{proof}

\noindent\textbf{Theorem A.4.} \emph{Let $\alpha,\beta>-1$ and $\phi\in F_N^{\psi}$. Then we have
\begin{equation*}
  \|{\delta}^{m}_{\psi}\phi\|_{\varpi^{\alpha+m,\beta+m}}\lesssim N^m\|\phi\|_{\varpi^{\alpha,\beta}}.
\end{equation*}}
\begin{proof}~Let
\begin{equation*}
  \phi(x)=\sum_{k=0}^N \widehat{\phi}_kJ_{k,\psi}^{\alpha,\beta}(x),\quad
\widehat{\phi}_k=\frac{\int_{a}^b \phi(x) J_{k,\psi}^{\alpha,\beta}(x)
\varpi^{\alpha,\beta}(x){\rm d}x}
{\|J_{k,\psi}^{\alpha,\beta}\|^2_{\varpi^{\alpha,\beta}}}.
\end{equation*}
Then
\begin{equation*}
  \|\phi\|^2_{\varpi^{\alpha,\beta}}=\sum_{k=0}^N \widehat{\phi}^2_k\gamma_k^{\alpha,\beta},
\end{equation*}
and
\begin{equation*}
\begin{split}
  &\|{\delta}^{m}_{\psi}\phi\|^2_{\varpi^{\alpha+m,\beta+m}}=\sum_{k=m}^N \widehat{\phi}^2_k\left(d_{k,m}^{\alpha,\beta}\right)^2\gamma_{k-m}^{\alpha+m,\beta+m}
=\sum_{k=m}^N \widehat{\phi}^2_k \gamma_{k}^{\alpha,\beta} \left(d_{k,m}^{\alpha,\beta}\right)^2\frac{\gamma_{k-m}^{\alpha+m,\beta+m}}
{\gamma_k^{\alpha,\beta}}\\
&\leq \kappa^{2m}\frac{(2N+\alpha+\beta+1)N!\Gamma(N+m+\alpha+\beta+1)}
{(2(N-m)+\alpha+\beta+1)(N-m)!\Gamma(N+\alpha+\beta+1)}\sum_{k=m}^N \widehat{\phi}^2_k \gamma_{k}^{\alpha,\beta}\\
&\lesssim N^{2m}\|\phi\|^2_{\varpi^{\alpha,\beta}}.
\end{split}
\end{equation*}
The proof is completed.
\end{proof}


\bigskip
\end{document}